\documentclass[final]{elsarticle}
\usepackage{jcomp}
\usepackage{framed,multirow}

\usepackage{amssymb}
\usepackage{latexsym}

\usepackage{amsmath,epsfig,cmmib57}
\usepackage{array}
\usepackage[T1]{fontenc}
\usepackage[latin9]{inputenc}
\usepackage{geometry}
\usepackage{times}
\usepackage{subcaption}
\usepackage{graphicx}
\usepackage{amsthm}
\usepackage{algorithm}
\usepackage[noend]{algpseudocode}
\usepackage{booktabs,graphicx,makecell,siunitx,eqparbox}
\usepackage{epic,eepic}
\usepackage{url}
\usepackage{xcolor}
\definecolor{newcolor}{rgb}{.8,.349,.1}

\newtheorem{theorem}{THEOREM}[section]

\newtheorem{assumption}[theorem]{Assumption}

\newcommand{\bsym}{\boldsymbol}

\journal{Journal of Computational Physics}

\begin{document}

\verso{Hee Jun Yang and Hyea Hyun Kim}

\begin{frontmatter}

\title{Iterative algorithms for partitioned neural network approximation
to partial differential equations}

\author[1]{Hee Jun \snm{Yang}}
\ead{yhjj109@khu.ac.kr}

\author[2]{Hyea Hyun \snm{Kim}\corref{cor1}\fnref{fn1}}
\ead{hhkim@khu.ac.kr}
\cortext[cor1]{Corresponding author}
%\fntext[fn1]{Hyea Hyun Kim's research is supported by the National Research Foundation of Korea (NRF) grants funded by %--- and by ---.}

% \address[1]{Affiliation 1, Address, City and Postal Code, Country}
\address[1]{Department of Mathematics, Kyung Hee University, Seoul, Republic of Korea}
\address[2]{Department of Applied Mathematics and Institute of Natural Sciences, Kyung Hee University, Yongin, Republic of Korea}

\begin{keyword}
% MSC codes here, in the form: \MSC code \sep code
% or \MSC[2008] code \sep code (2000 is the default)
% \MSC 41A05\sep 41A10\sep 65D05\sep 65D17
% Keywords
\KWD partitioned neural network approximation \sep partial differential equation \sep iterative algorithm \sep domain decomposition \sep parallel computing
\end{keyword}

\begin{abstract}
To enhance solution accuracy and training efficiency in neural network approximation
to partial differential equations,
partitioned neural networks can be used as a solution surrogate
instead of a single large and deep neural network
defined on the whole problem domain.
In such a partitioned neural network approach,
suitable interface conditions or subdomain boundary conditions
are combined to obtain a convergent approximate solution.
However, there has been no rigorous study on the convergence
and parallel computing enhancement on the partitioned neural network approach.
In this paper, iterative algorithms are proposed to address these issues.
Our algorithms are based on classical additive Schwarz domain decomposition methods.
Numerical results are included to show the performance of the proposed iterative algorithms.
\end{abstract}

\end{frontmatter}

%\linenumbers

%%%%%%%%%%%%%%%%%%%%%%%%%%%%%%

\section{Introduction}
With the success of deep learning technologies in many application areas, recently,
there have been developed many successful approaches to solve partial differential equations using
deep neural networks, see~\cite{weinan2017,sirignano2018,raissi2019,long2019}.
The advantage of these new approaches is that they can be used for partial differential equations
without much concern on the shape  or the dimension of the problem domain.
On the other hand, the approximate solution accuracy highly depends on hyper parameter settings,
such as, the network's depth and width, training data, activation functions, the learning rate,
and the optimization method. In addition, long parameter training time
makes the neural network solution very expensive.

There have been some successful attempts to alleviate these difficulties by introducing
partitioned neural networks as a solution surrogate, see \cite{Cpinn,Xpinn,FBpinn}.
In \cite{Cpinn,Xpinn}, partitioned neural networks are formed on a non-overlapping
subdomain partition of the problem domain, where each local neural network approximates
the solution in each subdomain and
additional continuity conditions on local neural network solutions are included
to the loss function
to obtain a convergent approximate solution.
In \cite{FBpinn}, a similar but different form of partitioned neural networks is proposed
based on an overlapping subdomain partition. In that approach, window functions are introduced
and they are used to form a global function, called FBPINNs (Finite Basis Physics Informed Neural Networks),
that is a sum of localized neural network functions.
Such a global function is used as a solution surrogate in their approach.
Differently to \cite{Cpinn,Xpinn}, no additional interface condition appears in the FBPINN method and
such a new approach was shown to be very effective for high oscillation model problem defined in a large problem domain.
In all these previous approaches, the local neural network parameters are updated every training epoch
and such an update procedure requires neighboring neural network's parameter information.
Since the number of training epochs can become easily more than several hundreds of thousand in many application problems, such a heavy communication cost makes the method inefficient under the parallel computing environment.

In this work, we propose iteration methods to address the heavy communication cost problem in the above mentioned partitioned neural network approaches. In our proposed iteration methods, the communication between neighboring subdomains happens every outer iteration but not every training epoch, and the proposed methods thus reduce the communication cost greatly.
Our iteration methods are based on the classical additive Schwarz algorithm.
The additive Schwarz algorithm is one of the most successful and well-known domain decomposition algorithms.
Domain decomposition algorithms are
widely used as fast solution procedures of algebraic equations arising from discretization of partial differential equations.
The original algebraic equations are restricted and solved in each subdomain combined with an iterative procedure.
The resulting solution for the original algebraic equations is then obtained from the iterative procedure.
In such approaches, the convergence often gets slow as more subdomains are introduced.
To accelerate the convergence, a global coarse problem is combined in the iterative procedure.
We refer \cite{TW-Book} for a general introduction to domain decomposition algorithms.

In our proposed iterative algorithms, we only consider partitioned neural networks built on
overlapping subdomain partitions.
We first propose an one-level iterative algorithm where each local neural network solves a local problem
at each outer iteration given the previous iterate.
We also provide a concrete convergence analysis for the proposed one-level algorithm.
The convergence in our one-level algorithm gets slower as more subdomains are in the partition,
that is a well-known and common convergence property in the one-level domain decomposition algorithms.
We thus include a coarse component in the partitioned neural networks to make
the iteration convergence robust to the increase of the number of subdomains.
For the proposed two-level iterative algorithm, we also prove a convergence result that
is independent of the number of subdomains.
We note that in previous pioneering studies~\cite{li2019,li2020-pro},
a similar idea is used but there has been no concrete convergence study
on the convergence of the iterative scheme.

Preliminary versions of this paper were appeared in \cite{DD26-kim-yang,arXiv-yang-kim-2022,DD27-kim-yang}.
In \cite{DD26-kim-yang,DD27-kim-yang}, the additive Schwarz algorithms are proposed and tested numerically
and in \cite{arXiv-yang-kim-2022} convergence analysis and some numerical results are presented.
In our current work, we improve the convergence results in the two-level method by proposing partitioned
neural networks combined with a partition of unity functions and we also present more extensive numerical
results for the proposed algorithms. For the completeness of the paper, we included the convergence analysis
of the additive Schwarz algorithms that has been carried out in our previous work~\cite{arXiv-yang-kim-2022}.
As related works, we also refer to \cite{heinlein2021combining}, where combination of machine learning and domain decomposition methods are reviewed and discussed, and to \cite{chung2021learning}, where a learning algorithm is developed to learn effective coarse basis in the adaptive BDDC domain decomposition methods.
Among many successful neural network approaches, we will simply use the PINN (Physics Informed Neural Network) method
in \cite{raissi2019}, when finding the approximate solutions for the local and coarse problems in our
proposed iterative scheme.
Our proposed methods can be extended to
other types of neural network approximation as well, for example those in \cite{weinan2017,sirignano2018,long2019}.

This paper is organized as follows. In Section~\ref{sec:pinn}, we introduce the PINN method for solving
partial differential equations and in Section~\ref{sec:ASM} we review the additive Schwarz method with
a relaxation parameter and its convergence analysis in a Hilbert space.
In Section~\ref{sec:dd:pinn}, we propose an one-level iterative scheme for
the partitioned neural network approximate solution.
The iterative scheme is further extended into a two-level iterative method by including a coarse neural network to
the partitioned neural network approximation.
For the proposed iterative methods, their convergence results
are proved with an assumption on local and coarse neural network solution errors.
In Section~\ref{sec:numerics}, numerical results are presented for test examples
and conclusions are given in Section~\ref{sec:conclude}.

%%%%%%%%%%%%%%%%%%%%%%%%%%%%%%

\section{Physics informed neural network (PINN) method}\label{sec:pinn}
Among many successful neural network approximation methods for solutions of partial differential equations,
we introduce the physics-informed neural network (PINN) method
that is developed based on minimizing the residual loss of the differential equations, see \cite{raissi2019}.
We first consider a general differential equation with a boundary condition,
\begin{equation}\label{eq:diff_eq}
\begin{aligned}
\mathcal{L}(u)&=f, \quad \text{in } \Omega, \\
\mathcal{B}(u)&=g, \quad \text{on } \partial\Omega,
\end{aligned}
\end{equation}
where $\mathcal{L}$ is a differential operator defined for a function $u$ and $\mathcal{B}$ describes a given boundary condition on $u$, and $f,g$ are given functions.
We assume that the model problem in \eqref{eq:diff_eq} is well-posed and the solution $u$ exists.

The solution $u$ in \eqref{eq:diff_eq} is approximated by a fully connected neural network function, $U({\bsym x};\theta)$,
$$U({\bsym x};\theta)=  W_L \sigma (W_{L-1} \sigma (\cdots \sigma( W_1 {\bsym x} + b_1 ) \cdots ) + b_{L-1} )
+ b_L,$$
where $\theta=(W_1,W_2,\cdots,W_L,b_1,b_2,\cdots,b_L)$ and $L$ denotes the number of layers.
The parameters $W_k \in R^{n_{k} \times n_{k-1} }$ and $b_k \in R^{n_{k}}$ are set with
their dimension as follows: For $n_0$ and $n_L$, they are set to
the dimension of the input and the dimension of the output and for $n_k$ in the $k$-th hidden layer, $k=1,\cdots,L-1$, they can be set freely depending on the required accuracy and the complexity of the model problem.
We also note that the activation function is applied on the output of each hidden layer
and is denoted by $\sigma(x)$ in the above expression, see also Fig.~\ref{fig:neural_network}
for an example of a fully connected neural network.
\begin{figure}[ht!]
	\centering
	\includegraphics[width=0.6\textwidth]{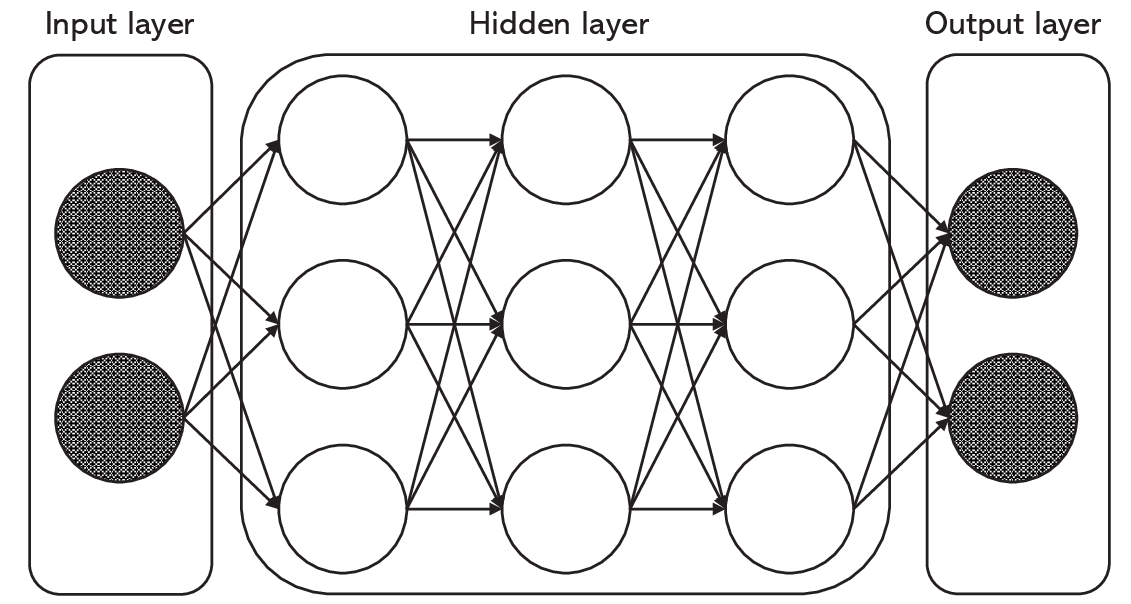}
	\caption{An example of a fully connected neural network with three hidden layers and three neurons per hidden layer.}
	\label{fig:neural_network}
\end{figure}

The parameters $\theta$ in $U({\bsym x};\theta)$ are then determined
to minimize the cost function $\mathcal{J}_X(\theta)$ consisting of two terms,
\begin{align}\label{loss:two-terms}
\mathcal{J}_X(\theta)=w_{I} \mathcal{J}_{X_\Omega}(\theta)+w_{B} \mathcal{J}_{X_{\partial\Omega}}(\theta),
\end{align}
where
\begin{align*}
\mathcal{J}_{X_{\Omega}}(\theta)=\frac{1}{|X_{\Omega}|}
\sum_{ {\bsym x} \in X_{\Omega} }|\mathcal{L}(U({\bsym{x}};\theta))-f({\bsym{x}})|^2
\end{align*}
and
\begin{align*}
\mathcal{J}_{X_{\partial\Omega}}(\theta)=\frac{1}{|X_{\partial\Omega}|}\sum_{ {\bsym{x}} \in X_{\partial\Omega}}|\mathcal{B}(U({\bsym{x}};\theta))-g({\bsym{x}})|^2.
\end{align*}
In the above, $X_{D}$ denotes the collection of points chosen from the region $D$ and $|X_{D}|$ denotes
the number of points in the set $X_D$, and $w_I$ and $w_B$ are weight factors associated
with the differential equation loss term and the boundary condition loss term, respectively.
The loss functions $\mathcal{J}_{X_\Omega}(\theta)$ and $\mathcal{J}_{X_{\partial\Omega}}(\theta)$
are designed so that the optimized neural network $U({\bsym{x}};\theta)$ satisfies
the equations in \eqref{eq:diff_eq}, derived from physics laws.
We note that $\mathcal{J}_{X_\Omega}(\theta)$ and $\mathcal{J}_{X_{\partial\Omega}}(\theta)$
approximate
$$\mathcal{J}_{\Omega}(\theta)=\int_{\Omega} (\mathcal{L}(U({\bsym{x}};\theta))-f({\bsym{x}})^2 \, d{\bsym{x}}
\quad \text{and} \quad
\mathcal{J}_{\partial \Omega}(\theta)=\int_{\partial \Omega} (U({\bsym{x}};\theta)-g({\bsym{x}}))^2 \, ds({\bsym{x}}),$$
and the accuracy and training performance in the neural network approximation $U({\bsym{x}},\theta)$
depend on the choice of the data sets $X_{\Omega}$ and $X_{\partial \Omega}$.
There are other options for the cost function, for an example, when $\mathcal{L}(u)=-\triangle u$
the following energy-based cost function was considered in \cite{yu2018deep},
$$\mathcal{J}(\theta)=\frac{1}{2} \int_{\Omega} |\nabla U({\bsym{x}};\theta)|^2 \, d{\bsym{x}}
-\int_{\Omega} f(x) U({\bsym{x}};\theta) \, d{\bsym{x}} +  \int_{\partial \Omega} |U({\bsym{x}};\theta)-g({\bsym{x}})|^2 \, ds({\bsym{x}}).$$

Regarding the accuracy of the neural network approximation solution $U({\bsym{x}};\theta)$,
it depends on the following hyper parameter settings, the number of layers $L$,
the number of nodes in each hidden layer, the training data sets $X_{\Omega}$, $X_{\partial \Omega}$,
the weight factors $w_I$, $w_B$, and optimization methods for the parameter training.
In a more detail, the error between the exact solution $u({\bsym{x}})$ and the resulting trained neural network $U({\bsym{x}};\theta_t)$
consists of the following three terms,
$$u({\bsym{x}})-U({\bsym{x}};\theta_t)=(u({\bsym{x}})-U({\bsym{x}};\theta_{a}))
+(U({\bsym{x}};\theta_a)-U({\bsym{x}};\theta_e))+(U({\bsym{x}};\theta_e)-U({\bsym{x}};\theta_t)),$$
where $U({\bsym{x}};\theta_a)$ denotes the neural network function with the minimum error to the exact solution,
and
$U({\bsym{x}};\theta_e)$ denotes the neural network function that optimizes the cost function $\mathcal{J}_X(\theta)$
in \eqref{loss:two-terms}.
The first term is called the approximation error, the second the estimation error, and the third
the optimization error.
The approximation error depends on the number of layers $L$, the number of nodes $n_k$ in each layer, and the activation function $\sigma(x)$, see \cite{hornik1989,hornik1991,pinkus1999}.
The estimation error depends on the number of training data, $|X_{\Omega}|$ and $|X_{\partial \Omega}|$,
the choice of data sets, $X_{\Omega}$, and $X_{\partial \Omega}$, and the weight factors $w_I$, $w_B$ in \eqref{loss:two-terms}.
The optimization error also depends on all the previously mentioned hyper parameters, the optimization method,
the learning rate, and the parameter initialization.

To reduce the approximation error, one needs to employ a massive layer or deeper network.
Since such a lager network is prone to the overfitting problem and the gradients of the cost function approaches vanishingly small as repeating multiplication, the estimation error and optimization error often get
larger and the parameter training process easily fails to give successful training results.
To overcome this limitation, partitioned neural networks can be used to approximate the solution $u({\bsym{x}})$.
As in earlier pioneering works~\cite{Cpinn,Xpinn,FBpinn}, the neural network approximation is formed
with localized neural network functions built on a partition of the problem domain and
the parameters in each local neural network function are trained by using a suitably formed cost function
with additional interface conditions.
Such a partitioned neural network approach allows a more flexible design of hyper parameter settings
and cost functions to give better training results than in the single neural network approach.
In addition, the localized nature of the partitioned neural network function makes
a fully parallel parameter training possible.
On the other hand, the parameter optimization process in the partitioned neural network approach~\cite{Cpinn,Xpinn,FBpinn}
needs heavy communication cost between neighboring subdomains, i.e., the neighboring subdomains exchange their parameter information every training epoch.
To resolve this heavy communication cost problem, we will develop iteration algorithms for training parameters
in the partitioned neural networks based on
the classical additive Schwarz method.
We note that in our proposed iteration algorithms
local parameters in each subdomain neural network are trained to approximate a local problem solution defined on the subdomain for the given current iterate, and these trained parameters are then communicated between the neighboring subdomains before proceeding to the next outer iteration.
The number of outer iteration is much smaller than the number of training epochs 
and the communication cost is greatly reduced in the proposed iteration algorithms,
that lead to a scalable parameter training procedure for the partitioned neural network approximation.

%Residual networks by \cite{He_2016_CVPR} are easier to optimize and can avoid the vanishing gradient problem.
%As seen in Figure~\ref{Fig1:resnet_fig}, the residual network consists of many blocks.
%The residual connection directly adds the current value $x$ to the output $h(x)$ of the current block and at the beginning of the next block, it gets $\sigma(h(x)+x)$ as the input,
%where $\sigma(x)$ denotes the activation function.
%In this case, the gradient values of the cost function do not vanish as proceeding the back-propagation step
%under the chain rule.
%
%
%\begin{figure}[]
%\centering
%\includegraphics[width=7cm,height=5.3cm]{resnet_fig.eps}
%\caption{Residual network: an example of one building block that consists
%of four wight layers with Tanh as the activation function.}
%\label{Fig1:resnet_fig}
%\end{figure}

\section{Additive Schwarz algorithm}\label{sec:ASM}
In this section, we introduce a classical additive Schwarz method as an iterative
procedure for finding the solution of the following
model elliptic problem in a domain $\Omega$,
\begin{equation}\label{model:elliptic}
\begin{split}
-\triangle u &= f \text{ in } \Omega,\\
u&=g \text{ on } \Omega.
\end{split}
\end{equation}
For that, an overlapping subdomain partition $\{\Omega_i\}_i$ of $\Omega$
with an overlapping width $\delta$ is introduced
and the iteration procedure is given as follows:
for a given $u^{(n)}$, the following problem in each subdomain $\Omega_i$ is solved for $u_i^{(n+1)}$,
\begin{equation}\label{local:elliptic}
\begin{split}
-\triangle u_i^{(n+1)}&=f \text{ in } \Omega_i,\\
%u_i^{(n+1)}&=u^{(n)} \text{ on } \partial \Omega_i,\\
u_i^{(n+1)}&=u^{(n)} \text{ in } \overline{\Omega} \setminus \Omega_i.
\end{split}
\end{equation}
After solving for $u_i^{(n+1)}$, the next iterate is obtained as
\begin{equation}\label{iterate:formula}
u^{(n+1)}=(1-N \tau ) u^{(n)} + \tau \sum_{i=1}^N u_i^{(n+1)},
\end{equation}
where $N$ denotes the number of subdomains and $\tau$ is a relaxation parameter.
Let $N_c$ be the maximum number of subdomains sharing the same geometric position in $\Omega$.
With $0< \tau \le 1/N_c$, the iterates $u^{(n)}$ converge to the solution $u$ of \eqref{model:elliptic}
under a suitably chosen space of functions, see~\cite{OS-Xu,TW-Book,OS-Martin,OS-Park}.

For the sake of completeness, we now present a more refined convergence analysis
for the above additive Schwarz method under a Hilbert space setting
by adapting the notations and ideas in \cite{badea2000additive}.

\subsection{Projection onto a Hilbert space}
Let $V$ be a Hilbert space with an inner product $(\cdot,\cdot)_V$ and
$a(\cdot,\cdot)$ be a bilinear form defined for $V \times V$, that is related to the model elliptic problem, i.e.,
$$a(v,w):=\int_{\Omega} \nabla v \cdot \nabla w \, dx. $$
%and its restriction to each subspace $V_i$ be
%$$a_i(v_i,w_i):=\int_{\Omega} \nabla v_i \cdot \nabla w_i \, dx,$$
%where the functions $v_i$ and $w_i$ are in a Hilbert space $V_i$ defined for $\Omega_i$.
We will show that the iterates $u^{(n)}$ in the above additive Schwarz scheme converge to the exact solution $u$
in the Hilbert space $V:=H_0^1(\Omega)$, where $H_0^1(\Omega)$ consists of square integrable functions
up to the first weak derivatives with the vanishing boundary value.
Accordingly, the local Hilbert space $V_i$ is given as $V_i:=H_0^1(\Omega_i)$.
When $v_i$ is in the space $V_i$, by extending it zero into $\Omega$, we can consider $v_i$
as a function in $V$.
In the presentation below, for simplicity
we assume that functions in $V_i$ are extended by zero so that
they also belong to the space $V$.

For the local problem in \eqref{local:elliptic}, we can form it as
$$a(u_i^{(n+1)}-u^{(n)},v_i)=f(v_i)-a(u^{(n)},v_i),\; \forall v_i \in V_i$$
and for the solution $u$, we obtain that
$$a(u,v) = f(v),\; \forall v \in V,$$
where
$$f(v):=\int_{\Omega} f v \, dx.$$

We then obtain
$$a(u_i^{(n+1)}-u^{(n)},v_i)=a(u-u^{(n)},v_i),\; \forall v_i \in V_i,$$
and
$$u_i^{(n+1)}-u^{(n)}=P_i (u-u^{(n)}),$$
where
$P_i$ is the projection onto the subspace $V_i$ defined as
$$a(P_i(v),v_i)=a(v,v_i),\; \forall v_i \in V_i.$$
We note that $u_i^{(n+1)}-u^{(n)}$ is in the space $V_i$.
The iterate $u^{(n+1)}$ in \eqref{iterate:formula} is then obtained as
$$u^{(n+1)}=u^{(n)}+\tau \sum_{i=1}^N (u_i^{(n+1)}-u^{(n)})=u^{(n)}+\tau \sum_{i=1}^N P_i (u-u^{(n)})$$
and the resulting error equation then follows
\begin{equation}\label{ASM:iter:err}
u-u^{(n+1)}=(I-\tau \sum_{i=1}^N P_i ) (u-u^{(n)}).
\end{equation}

\subsection{Convergence analysis}
We will now present the convergence analysis of the additive Schwarz method.
Our analysis is based on the derivation in \cite{badea2000additive}, where
the additive Schwarz algorithm is applied and analyzed for variational inequalities.
We first introduce the following two assumptions:

\begin{assumption}\label{assume1}
For $v$ in $V$, there exist $v_i$ in $V_i$ such that
$$v=\sum_{i=1}^N v_i \; \text{ and } \; \sum_{i=1}^N a(v_i,v_i) \le C_0 a(v,v),$$
where the constant $C_0$ is independent of $v$.
\end{assumption}

\begin{assumption}\label{assume2}
For $v_i$ in $V_i$ and $v_j$ in $V_j$, there exists $0 \le \mathcal{E}_{ij} \le 1$ such that
$$a(v_i,v_j) \le \mathcal{E}_{ij} a(v_i,v_i)^{1/2} a(v_j,v_j)^{1/2}.$$
\end{assumption}

It is well known that
the above two assumptions were introduced for analysis of the classical Schwarz methods,
see \cite{TW-Book}.

\begin{theorem}\label{thm:conv:factor}
Under the Assumptions~\ref{assume1} and \ref{assume2}, the iterates $u^{(n)}$ in the Hilbert space satisfy
$$a(u-u^{(n+1)},u-u^{(n+1)}) \le  \left( 1-\frac{2}{2+C_0} \tau + \| \mathcal{E} \|_{2}^2 \tau^2 \right) a(u-u^{(n)},u-u^{(n)}),$$
where $C_0$ is the constant in Assumption~\ref{assume1} and $\| \mathcal{E} \|_{2}$ is the 2-norm
for the matrix $\mathcal{E}$ with $\mathcal{E}_{ij}$ in Assumption~\ref{assume2} as its entries.
\end{theorem}
\begin{proof}
For $v$ in $V$, we take its decomposition $\{v_i\}$ such that $\sum_{i=1}^N v_i=v$ and satisfies
Assumption~\ref{assume1}.
We then consider
\begin{align}
\sum_{i=1}^N a(P_i v, v_i) & \le \sum_{i=1}^N a(P_i v, P_i v)^{1/2} a(v_i,v_i)^{1/2} \nonumber \\
& \le \left( \sum_{i=1}^N a(v,P_i v) \right)^{1/2} \left( \sum_{i=1}^N a(v_i,v_i) \right)^{1/2} \nonumber \\
& \le \left( \sum_{i=1}^N a(v,P_i v) \right)^{1/2} \sqrt{C_0} a(v,v)^{1/2} \nonumber \\
& \le \frac{C_0}{2} \sum_{i=1}^N a(v,P_i v) + \frac{1}{2} a(v,v) \label{ineq1},
\end{align}
where we used the Assumption~\ref{assume1} in the third inequality.
We also note that
\begin{align}
a(v,v) &=\sum_{i=1}^N a(v,v_i)=\sum_{i=1}^N a(v,v_i-P_i v) + \sum_{i=1}^N a(v,P_i v) \nonumber \\
&=\sum_{i=1}^N a(P_i v, v_i-P_i v)+\sum_{i=1}^N a(v,P_i v) \le \sum_{i=1}^N a(P_i v, v_i) + \sum_{i=1}^N a(v,P_i v) \nonumber \\
&\le \frac{C_0}{2} \sum_{i=1}^N a(v,P_i v) + \frac{1}{2} a(v,v) + \sum_{i=1}^N a(v,P_iv),\label{ineq2}
\end{align}
where we used the inequality~\eqref{ineq1} in the above last inequality.
By \eqref{ineq2}, we obtain that
\begin{equation}\label{ineq3}
a(v,v) \le (2 + C_0 ) \sum_{i=1}^N a(v,P_iv).
\end{equation}

We let $P=\sum_{i=1}^N P_i$ in \eqref{ASM:iter:err} and consider
\begin{align}\label{three-parts}
&a(u-u^{(n+1)},u-u^{(n+1)})\nonumber \\
&=a((I-\tau P) (u-u^{(n)}), (I-\tau P) (u-u^{(n)})) \nonumber \\
&=a(u-u^{(n)},u-u^{(n)})-2\tau a(u-u^{(n)},P(u-u^{(n)}))+\tau^2 a(P(u-u^{(n)}),P(u-u^{(n)})).
\end{align}

For the second term in \eqref{three-parts}, letting $v=u-u^{(n)}$ in \eqref{ineq3}, we obtain that
\begin{align}
a(u-u^{(n)},P(u-u^{(n)}))&=\sum_{i=1}^N a(u-u^{(n)},P_i (u-u^{(n)}))\nonumber\\
&\ge \frac{1}{2+C_0} a(u-u^{(n)},u-u^{(n)}).\label{second-term}
\end{align}

For the third term in \eqref{three-parts}, by using the Assumption~\ref{assume2},
\begin{align}
&a(P(u-u^{(n)}),P(u-u^{(n)})) \nonumber\\
&=\sum_{i,j} a(P_i(u-u^{(n)}),P_j(u-u^{(n)})) \nonumber \\
&\le \sum_{i,j} \mathcal{E}_{ij} a(P_i (u-u^{(n)}), P_i (u-u^{(n)}))^{1/2}a(P_j (u-u^{(n)}), P_j (u-u^{(n)}))^{1/2}
\nonumber \\
&\le \| \mathcal{E} \|_2 \| {\bf q} \|_2^2, \label{third-term:0}
\end{align}
where $\mathcal{E}$ denotes the matrix with its entries $\mathcal{E}_{ij}$, ${\bf q}$ is the vector
with its entries $a(P_i (u-u^{(n)}), P_i (u-u^{(n)}))^{1/2}$, and $\|\cdot\|_2$ denotes the matrix or vector 2-norm.
We note that
\begin{align}
\| {\bf q} \|_2^2 &=\sum_{i=1}^N a(P_i(u-u^{(n)}),P_i(u-u^{(n)}))\nonumber \\
&=\sum_{i=1}^N a(P_i(u-u^{(n)}),u-u^{(n)})\nonumber \\
&=a(P(u-u^{(n)}),u-u^{(n)})\nonumber \\
&\le a(P(u-u^{(n)}),P(u-u^{(n)}))^{1/2} a(u-u^{(n)},u-u^{(n)})^{1/2}. \label{ineq:q}
\end{align}
Combining \eqref{third-term:0} with \eqref{ineq:q}, we obtain that
\begin{equation}\label{third-term}
a(P(u-u^{(n)}),P(u-u^{(n)})) \le \| \mathcal{E} \|_{2}^2\, a(u-u^{(n)},u-u^{(n)}).
\end{equation}

Combining \eqref{three-parts} with the inequalities in \eqref{second-term} and \eqref{third-term},
we finally obtain the resulting estimate,
\begin{align}
&a(u-u^{(n+1)},u-u^{(n+1)})\nonumber \\
&\le a(u-u^{(n)},u-u^{(n)}) - \frac{2}{2+C_0} \tau a(u-u^{(n)},u-u^{(n)})
+\|\mathcal{E}\|_2^2 \tau^2 a(u-u^{(n)},u-u^{(n)})\nonumber \\
&=\left( 1-\frac{2}{2+C_0} \tau + \| \mathcal{E} \|_2^2 \tau^2 \right) a(u-u^{(n)},u-u^{(n)}). \label{result:error:factor}
\end{align}
\end{proof}

For the matrix norm $\| \mathcal{E} \|_2$, we can obtain that
$$\| \mathcal{E} \|_2 \le N_c,$$
where $N_c$ is the maximum number of subdomains sharing the same position $x$ in $\Omega$, i.e.,
the number of colors in the coloring arguments, see~\cite[Lemma 3.11]{TW-Book}.
Letting
$$R(\tau)=1-\frac{2}{2+C_0} \tau +N_c^2 \tau^2,$$
the error reduction factor is bounded by $R(\tau)$;
\begin{equation}\label{error:reduction}
a(u-u^{(n+1)},u-u^{(n+1)}) \le R(\tau) a(u-u^{(n)},u-u^{(n)}).
\end{equation}
We note that when $\tau=1/(N_c^2 (2+C_0))$, $R(\tau)$ achieves its minimum
value,
\begin{equation}\label{error:reduction:mim}
\min_{\tau} R(\tau)=1-\frac{1}{N_c^2 (2+C_0)^2}.
\end{equation}

We now introduce a partition of unity functions, $\{ \phi_i(x)\}_i$, such that
$$\sum_i \phi_i(x)=1$$
and $\phi_i(x)( \ge 0)$ vanishes outside the subdomain $\Omega_i$.
Let $V$ be a subspace of $H_0^1(\Omega)$.
Regarding to the constant $C_0$ in Assumption~\ref{assume1}, for $v$ in $V$ we can obtain, see~\cite[Lemma 3.11, page 69]{TW-Book},
\begin{equation}\label{estimate:C0}
\sum_{i=1} a(v_i,v_i) \le C N_c \left( 2 + (N+1) \frac{H}{\delta} \right) a(v,v),
\end{equation}
where $v_i=\phi_i v$ with $\phi_i$ being the partition of unity functions,
$\delta$ is the maximum overlapping width in the subdomain partition,
$H$ is the maximum subdomain diameter,
$N$ is the number of subdomains in the partition,
and $C$ is a constant independent of the function $v$, and
the aforementioned parameters $\delta$, $H$, and $N$.
With $C_0=C N_c \left( 2 + (N+1) H / \delta \right)$, the Assumption~\ref{assume1} holds for $v$ in $V$.

When $H/\delta$ and $N_c$ are fixed, as increasing $N$ the constant $C_0$ follows the growth of
$N_c N H/\delta$,
$$C_0 \simeq N_c N \frac{H}{\delta}$$
and putting it into the error reduction rate, we obtain;
As $N$ increases, with given fixed $H / \delta$ and $N_c$, the error reduction rate follows the growth,
\begin{equation}\label{final:reduction:estimate}
\min_\tau R(\tau) \simeq 1- \frac{1}{N_c^4 N^2 ({H}/{\delta})^2},
\end{equation}
where we used
$$N_c^2 (2+C_0)^2 \simeq N_c^2 (N_c N H/\delta)^2 \simeq N_c^4 N^2 (H/\delta)^2.$$

By \eqref{final:reduction:estimate}, we can conclude that
the error reduction rate increases to the value 1 as $N$ increases.
The convergence of the proposed additive Schwarz algorithm gets slower as more subdomains are in
the partition with a fixed ratio of the overlapping width, i.e., $H/\delta$ being fixed.
We note that by including a coarse component $v_0$ in a coarse subspace $V_0$ of $V$ to the decomposition
we can obtain the constant $C_0$ to be independent of the number of subdomains $N$.
With the addition of the coarse component, the above additive Schwarz iteration method can be extended
into a two-level algorithm to speed up the iteration convergence as more subdomains in the partition.

\section{Iterative algorithms by neural network approximate solutions}\label{sec:dd:pinn}
In this section, we will propose iteration algorithms for partitioned neural network approximation
to the model elliptic problem introduced in the previous section.
Our iteration algorithms are based on the additive Schwarz algorithm.
In the following we will propose one-level and two-level iterative algorithms.

\subsection{One-level algorithm}\label{subsec:one-level}
In this subsection, we introduce the one-level iterative algorithm for partitioned neural network approximation
to the model elliptic problem.
In the iteration algorithm, starting with an initial neural network $U^{(0)}$, we will solve local problems
with the given boundary value by $U^{(0)}$ and update the iterate $U^{(1)}$ to proceed the next iterate.
In general, at the outer iteration $n+1$, our partitioned neural network solution is obtained as,
\begin{equation}\label{def:hatU}
\widehat{U}^{(n+1)}(x):=\frac{1}{|s(x)|} \sum_{i \in s(x)} U_i(x;\theta_i^{(n+1)}),
\end{equation}
where $s(x)$ is the set of subdomain indices sharing $x$, $|s(x)|$ is the number of elements in the set,
and $U_i(x;\theta_i^{(n+1)})$ is the trained local neural network solution with the given boundary value $U^{(n)}$.
The iterate $U^{(n)}$ is then updated into $U^{(n+1)}$ to proceed the next iteration:
\begin{equation}\label{one-level:update:Un}
U^{(n+1)}(x)=(1-\tau |s(x)| ) U^{(n)}(x) + \tau |s(x)| \widehat{U}^{(n+1)}(x),
\end{equation}
where $\tau$ is the relaxation parameter, introduced in the additive Schwarz algorithm,
and $\widehat{U}^{(n+1)}(x)$ is the partitioned neural network solution.

The one-level algorithm is then listed as follows:

{\bf Algorithm 1: One-level method} (input: $U^{(0)}$, output: $\widehat{U}^{(n+1)}$)

{\bf Step 0}: Let $U^{(0)}(x)$ be given and $n=0$.

{\bf Step 1}: Find $\theta_i^{(n+1)}$ in $U_i(x;\theta_i^{(n+1)})$ for the local problem in each subdomain $\Omega_i$,
\begin{equation}\label{local:pb}
\begin{split}
-\triangle u&=f \;\text{ in } \Omega_i, \\
u&=U^{(n)}\; \text{ on } \partial \Omega_i.
\end{split}
\end{equation}

{\bf Step 2}: Update $U^{(n+1)}$ at each data set $X_{\partial \Omega_i}$, see \eqref{one-level:update:Un}.

{\bf Step 3}: Go to {\bf Step 1} with $n=n+1$ or set the output as $\widehat{U}^{(n+1)}$ if the stopping condition is met.

\vspace{0.2cm}

We note that in Step 1 of Algorithm 1 the parameter $\theta_i^{(n+1)}$ is optimized for the following cost function,
$$\mathcal{J}_{i,X}(\theta):=w_{I,i} \frac{1}{|X_{\Omega_i}|} \sum_{x \in X_{\Omega_i}}(\triangle U_i(x;\theta)+f(x))^2
+w_{B,i} \frac{1}{|X_{\partial \Omega_i}|} \sum_{x \in X_{\partial \Omega_i}} (U_i(x;\theta)-U^{(n)}(x))^2,$$
where $X_{\Omega_i}$ and $X_{\partial \Omega_i}$ are training data sets for the differential equation
and the boundary condition, respectively, and $w_{I,i}$ and $w_{B,i}$ are weight factors.
In Step 2, we update the values of $U^{(n+1)}(x)$ only at the data set $X_{\partial \Omega_i}$, $i=1,\cdots,N$, that are needed for Step 1 in the next outer iteration.
We also note that we need the data communication between neighboring subdomains when updating the values $U^{(n+1)}(x)$.

We will now prove that the partitioned neural network $\widehat{U}^{(n)}$ converge to the exact solution $u$.
For that, we will first show that $U^{(n)}$ converge to the exact solution $u$.
We recall the iteration formula in \eqref{iterate:formula} and the local problem in \eqref{local:elliptic}.
By extending the local neural network solutions to the domain $\Omega$,
we set the functions $U^{(n+1)}_i(x)$ as
\begin{equation}\label{def:Ui}
U^{(n+1)}_i(x)=U_i(x;\theta_i^{(n+1)}),\, \forall x \in \Omega_i,\quad
U^{(n+1)}_i(x)=U^{(n)}(x),\, \forall x \in \Omega \setminus \Omega_i
\end{equation}
and using them we can express the iterate $U^{(n+1)}$ as
\begin{equation}\label{U:iter}
U^{(n+1)}(x)=(1-\tau N)U^{(n)}(x)+\tau \sum_{i=1}^N U_i^{(n+1)}(x).
\end{equation}

We now introduce $|\cdot|_{1,\Omega_i}$ as the seminorm in the Hilbert space $V_i=H^1(\Omega_i)$
associated with the bilinear form $a_i(u_i,v_i):=\int_{\Omega_i} \nabla u_i \cdot \nabla v_i \, dx$
and
we introduce the following assumption on the neural network solution:
\begin{assumption}\label{assume3}
Let $\tilde{u}_i$ and $\tilde{U}_i$ be the Hilbert space solution and the neural network solution to the local problem~\eqref{local:pb}.
The error between $\tilde{U}_i$ and the Hilbert space solution $\tilde{u}_i$
is bounded by a sufficiently small $\epsilon$:
$$|\tilde{u}_i(x)-\tilde{U}_i(x)|_{1,\Omega_i} \le \frac{\epsilon}{N},$$
where $N$ denotes the number of subdomains in the partition.
\end{assumption}
We note that local Hilbert space solution $\tilde{u}_i$ satisfy
$\tilde{u}_i-U^{(n)}=P_i(u-U^{(n)})$ with $P_i$ as the projection onto the Hilbert space $H_0^1(\Omega_i)$.
For the neural network solution $\tilde{U}_i$, its error to the Hilbert space solution $\tilde{u}_i$
is affected by the approximation error, estimation error, and optimization error.
Under our settings, as $N$ gets larger, the subdomain size gets smaller and the solution $\widetilde{u}_i$
in such a smaller subdomain has relatively less variations, and it thus can be well approximated
by a neural network function with a moderate number of parameters, in practice.
With this heuristic explanation, we can say that the approximation error satisfies the above assumption.
We note that the resulting neural network solution $\tilde{U}_i$ is also affected by the estimation error and optimization error,
and such additional errors need to be well-controlled by setting suitable hyper parameters, so as to
make the Assumption~\ref{assume3} valid for the neural network solution $\tilde{U}_i$.

We now prove that $U^{(n)}$ converge to the exact solution $u(x)$ in the following theorem:
\begin{theorem}\label{thm:conv:one-level:nn}
Under the Assumptions~\ref{assume1}, \ref{assume2}, and \ref{assume3}, the iterates $U^{(n)}$
satisfy the following error estimate,
$$|u-U^{(n+1)}|_1 \le |u-u^{(n+1)}|_1 + \frac{1}{1-C_1} \epsilon,$$
where $\epsilon$ is the local neural network solution error in Assumption~\ref{assume3}
and the constant $C_1$ is the error reduction factor in the one-level additive Schwarz algorithm
for the iterates $u^{(n)}$ that are obtained from the same initial $u^{(0)}=U^{(0)}$.
\end{theorem}
\begin{proof}
We will show the following estimate for $U^{(n+1)}$:
\begin{align}
|u-U^{(n+1)} |_1 &\le |u-u^{(n+1)}|_1 + |u^{(n+1)}-U^{(n+1)}|_1 \nonumber \\
& \le |u-u^{(n+1)}|_1 + C \epsilon, \label{error:u-Un}
\end{align}
where the constant $C$ is independent of $n$.

It then suffices to show the estimate in \eqref{error:u-Un}, i.e.,
$$|u^{(n+1)}-U^{(n+1)}|_1 \le C \epsilon,$$
with $C=1/(1-C_1)$.
For the iterate $U^{(n)}$, we define
$$\tilde{u}^{(n+1)}=(1-\tau N) U^{(n)}+\tau \sum_{i=1}^N \tilde{u}_i^{(n+1)},$$
where $\tilde{u}_i^{(n+1)}$ are the local Hilbert space solutions obtained from the given previous iterate $U^{(n)}$. We also recall the iteration formula in \eqref{U:iter}.
From the Assumption~\ref{assume3} and the iteration formula in \eqref{U:iter}, we obtain
$$|\tilde{u}^{(n+1)}-U^{(n+1)}|_1 =\left|\tau \sum_{i=1}^N \left(\widetilde{u}_i^{(n+1)}(x)-U_i^{(n+1)}(x)\right)\right|_1
\le \tau \sum_{i=1}^N \frac{\epsilon}{N} \le \epsilon.$$
For the error $|u^{(n+1)}-U^{(n+1)}|_1$, we then obtain
\begin{align*}
|u^{(n+1)}-U^{(n+1)}|_1 & \le |u^{(n+1)}-\tilde{u}^{(n+1)}|_1 + | \tilde{u}^{(n+1)}-U^{(n+1)}|_1\\
&\le C_1 |u^{(n)}-U^{(n)}|_1 + \epsilon,
\end{align*}
where we used that the error between the two solutions ${u}^{(n+1)}$ and $\tilde{u}^{(n+1)}$ in the Hilbert space
corresponding to the different previous iterates $u^{(n)}$ and $U^{(n)}$ is bounded by, see \eqref{ASM:iter:err},
$$|u^{(n+1)}-\tilde{u}^{(n+1)}|_1 =|(I-\tau P)(u^{(n)}-U^{(n)})|_1\le C_1 |u^{(n)}-U^{(n)}|_1$$
with $C_1$ being the error reduction factor of the one-level additive Schwarz algorithm in the Hilbert space.

Applying similarly for $|u^{(n)}-U^{(n)}|_1$, we obtain that
$$|u^{(n)}-U^{(n)}|_1 \le C_1 |u^{(n-1)}-U^{(n-1)}|_1 + \epsilon,$$
and finally obtain that
\begin{align}
|u^{(n+1)}-U^{(n+1)}|_1 & \le (1+C_1+C_1^2+\cdots+C_1^{(n-1)})\epsilon+ C_1^{n} |u^{(1)}-U^{(1)}|_1 \nonumber \\
& \le (1+C_1+C_1^2+\cdots+C_1^n) \epsilon \nonumber \\
& = \frac{1-C_1^{n+1}}{1-C_1} \epsilon \nonumber \\
& \le \frac{1}{1-C_1} \epsilon. \label{error:un-Un}
\end{align}
In the above,
$$|u^{(1)}-U^{(1)}|_1 \le C_1 |u^{(0)}-U^{(0)}|_1+\epsilon = \epsilon,$$
since we initialized $u^{(0)}=U^{(0)}$.
We now proved the estimate in \eqref{error:u-Un} with the constant $C$ bounded by the value $1/(1-C_1)$.
This completes the proof.
\end{proof}

We will now show that the partitioned neural network iterates $\widehat{U}^{(n+1)}$ in \eqref{def:hatU}
converge to the exact solution $u$.
Recalling the iteration formula in \eqref{one-level:update:Un}, we obtain
\begin{equation*}
u(x)-\widehat{U}^{(n+1)}(x)=\left(1-\frac{1}{\tau |s(x)| } \right)(u(x)-U^{(n)}(x)) + \frac{1}{\tau |s(x)|} (u(x)-U^{(n+1)}(x)),
\end{equation*}
and
\begin{equation}\label{convergence:one-level}
|u(x)-\widehat{U}^{(n+1)}(x) | \le \frac{1}{\tau} | u(x)-U^{(n)}(x) | + \frac{1}{\tau} | u(x)-U^{(n+1)}(x) |,
\end{equation}
where we used $|s(x)|\ge 1$ and $0 < \tau |s(x) | \le 1$.
We note that in the one-level Schwarz algorithm the relaxation parameter $\tau$ is chosen to be less than or equal to a fraction of the maximum value of $|s(x)|$, i.e., $1/N_c$, with $N_c$ being the number of colors in the partition.

From the result in Theorem~\ref{thm:conv:one-level:nn}, the iterates $U^{(n+1)}$ satisfy
$$\| u-U^{(n+1)}\|_0 \le C_p |u-U^{(n+1)}|_1 \le C_p ( |u-u^{(n+1)}|_1 + \frac{1}{1-C_1} \epsilon),$$
where $\| \cdot \|_0$ denotes the $L^2$-norm and $C_p$ is the constant in the Poincar\'{e} inequality.
Combining the above estimate with \eqref{convergence:one-level}, the iterates $\widehat{U}^{(n+1)}$ thus satisfy
\begin{align*}
\| u-\widehat{U}^{(n+1)}\|_0 &\le \frac{1}{\tau} ( \|u-U^{(n)}\|_0 + \|u-U^{(n+1)}\|_0 )\\
&\le \frac{C_p}{\tau} ( |u-u^{(n)}|_1+|u-u^{(n+1)}|_1+\frac{2}{1-C_1} \epsilon),
\end{align*}
to give the following error estimate for $\widehat{U}^{(n+1)}$ in the $L^2$-norm:
\begin{theorem}\label{cor:conv:one-level:nn:uhat}
Under the Assumptions~\ref{assume1}, \ref{assume2}, and \ref{assume3}, the partitioned neural network iterates $\widehat{U}^{(n)}$
satisfy
$$\|u-\widehat{U}^{(n+1)}\|_0 \le \frac{C_p}{\tau} \left(|u-u^{(n)}|_1 + |u-u^{(n+1)}|_1+\frac{2}{1-C_1} \epsilon \right),$$
where $\epsilon$ is the local neural network solution error in Assumption~\ref{assume3} and the constant $C_1$ is the error reduction factor for the iterates $u^{(n)}$ in the one-level additive Schwarz algorithm.
\end{theorem}

\subsection{Two-level algorithm}\label{subsec:two-level}
In this subsection, in order to improve the convergence of the one-level algorithm
we propose a two-level algorithm by including a coarse
neural network solution $W_0(x;\theta^{(n+1)})$, that approximates the following coarse problem solution
$w(x)$ in a coarse subspace $V_0$ of $H_0^1(\Omega)$,
\begin{equation}\label{parallel:coarse:correction}
\begin{split}
-\triangle w &=f+\triangle U^{(n)}\; \text{ in } \Omega,\\
w&=0 \; \text{ on } \partial \Omega.
\end{split}
\end{equation}
Inclusion of such a coarse component in the two-level algorithm makes the iteration convergence
robust to the number of subdomains in the partition and such a property is important in
achieving scalable algorithms.
In the two-level algorithm the same local problems as in the one-level case are solved in each subdomain.
At each outer iteration, we obtain the resulting partitioned neural network solution as
\begin{equation}\label{parallel:two-level:sol}
\widehat{U}^{(n+1)}(x)=\frac{1}{|s(x)|} \left( W_0(x;\theta_0^{(n+1)}) + \sum_{i \in s(x)} U_i(x;\theta_i^{(n+1)})\right),
\end{equation}
that consists of the coarse neural network solution and the local neural network solutions
for the given iterate $U^{(n)}$.
The iterates $U^{(n+1)}$ in the two-level iteration formula are then undated as
\begin{equation}\label{parallel:two-level}
U^{(n+1)}=(1-N\tau) U^{(n)} + \tau \left( W_0(x;\theta_0^{(n+1)})+\sum_{i=1}^N U_i^{(n+1)}(x) \right),
\end{equation}
where $U_i^{(n+1)}(x)$ are defined as in \eqref{def:Ui}
and the above iteration formula is identical to
\begin{equation}\label{update:Un:two-level}
U^{(n+1)}(x)=(1-\tau |s(x)| ) U^{(n)}(x) + \tau |s(x)| \widehat{U}^{(n+1)}(x).
\end{equation}

The two-level algorithm is then listed as follows:

{\bf Algorithm 2: Two-level method} (input: $U^{(0)}$, output: $\widehat{U}^{(n+1)}$)

{\bf Step 0}: Let $U^{(0)}(x)$ be given and $n=0$.

{\bf Step 1-1}: Find $\theta_i^{(n+1)}$ in $U_i(x;\theta_i^{(n+1)})$ for each local problem,
\begin{equation*}
\begin{split}
-\triangle u&=f \;\text{ in } \Omega_i, \\
u&=U^{(n)}\; \text{ on } \partial \Omega_i.
\end{split}
\end{equation*}

{\bf Step 1-2}: Find $\theta_0^{(n+1)}$ in $W_0(x;\theta_0^{(n+1)})$ for the coarse problem,
\begin{equation*}\label{pb:coarse}
\begin{split}
-\triangle w&=f + \triangle U^{(n)} \; \text{ in } \Omega,\\
w&=0 \; \text{ on } \partial \Omega.
\end{split}
\end{equation*}

{\bf Step 2}: Update $U^{(n+1)}(x)$ at each data set $X_{\partial \Omega_i}$ and $\triangle U^{(n+1)}(x)$
at each data set $X_{\Omega}$, by using the iteration formula in \eqref{update:Un:two-level}.
%$$U^{(n+1)}(x)=(1-\tau |s(x)| ) U^{(n)}(x) + \tau |s(x)| \widehat{U}^{(n+1)}(x).$$

{\bf Step 3}: Go to {\bf Step 1-1} with $n=n+1$ or set the output as $\widehat{U}^{(n+1)}$
if the stopping condition is met.

\vspace{0.3cm}
We note that the problems in Step 1-1 and Step 1-2 can be solved in parallel for
the given current iterate $U^{(n)}$.
When updating the values in Step 2, the data communication
between local neural networks and the coarse neural network is needed
in addition to the data communication between neighboring local neural networks.

For the convergence analysis of the iterates $U^{(n)}$ in the two-level method, we can follow
the proof similarly as in the one-level case.
For the iteration formula in \eqref{parallel:two-level}, we consider the corresponding iteration formula
with the Hilbert space solutions,
\begin{equation}\label{parallel:two-level:hilbert}
u^{(n+1)}=(1-N\tau) u^{(n)} + \tau ( w_0^{(n+1)}(x)+\sum_{i=1}^N u_i^{(n+1)}(x)),
\end{equation}
where $u_i^{(n+1)}$ are defined as the same as in the previous one-level algorithm
and $w_0^{(n+1)}(x)$ is the Hilbert space solution in $V_0$ for the coarse problem,
\begin{equation*}\label{parallel:coarse:correction}
\begin{split}
-\triangle w &=f+\triangle u^{(n)}\; \text{ in } \Omega,\\
w&=0 \; \text{ on } \partial \Omega.
\end{split}
\end{equation*}
One can show the convergence of $u^{(n)}$ to $u$
in the Hilbert space $H^1(\Omega)$ with the following additional assumption on
the coarse subspace $V_0$ of $V(=H_0^1(\Omega))$:
\begin{assumption}\label{assume4}
For any $v$ in $V$, there exists $v_0$ in $V_0$ such that
$$\| v - v_0 \|_{0} \le C H | v-v_0 |_{1},$$
where $H$ denotes the largest subdomain diameter and $\| \cdot \|_{0}$ denotes the $L^2$-norm.
\end{assumption}
We note that the solution $w_0^{(n+1)}$ for the coarse problem in \eqref{parallel:coarse:correction}
is found in the Hilbert subspace $V_0$ of $V$ and $w_0^{(n+1)}=P_0 (u-u^{(n)})$, where $P_0$ is
the projection onto $V_0$ such that
$$a(P_0 v,v_0)=a(v,v_0),\; \forall v_0 \in V_0.$$
The error equation in the two-level Schwarz algorithm is then obtained as
$$u-u^{(n+1)}=(I-\tau \sum_{i=0}^N P_i) (u-u^{(n)})$$
and it gives the error reduction factor estimate in the two-level case, following similarly as in Theorem~\ref{thm:conv:factor},
$$a(u-u^{(n+1)},u-u^{(n+1)}) \le R(\tau) a(u-u^{(n)},u-u^{(n)}),$$
where
$$R(\tau)=1-\frac{2}{2+C_0} \tau + 2 ( \| \mathcal{E} \|_2^2+1 ) \tau^2.$$
The constant $C_0$ in the Assumption~\ref{assume1} is extended by including the coarse Hilbert space $V_0$
and the constant $C_0$ for such a case follows the growth $N_c H/\delta$ using
the Assumption~\ref{assume4} on the coarse Hilbert space $V_0$.
The error reduction factor $C_2$ is then bounded by $R(\tau)$.
The minimum value of $R(\tau)$ follows the growth of $1-1/(N_c C_0)^2$.
From this we then conclude that the minimum value of $R(\tau)$ follows the growth of $1-(1/N_c^4)(\delta/H)^2$,
which is independent of the number of subdomains $N$, and it shows that
the convergence rate $C_2$ in the two-level case
is robust to the increase of the number of subdomains in the partition.

Similarly as in the one-level method, by assuming that the coarse and local neural network solutions
approximate the Hilbert space solutions with a sufficiently small error $\epsilon_0$ and $\epsilon/N$, respectively,
we can show that
\begin{equation}\label{two:conv:H1}
|u-U^{(n+1)}|_1 \le |u-u^{(n)} |_1 + \frac{1}{1-C_2} (\epsilon_0+\epsilon),
\end{equation}
where $C_2$ is the rate of convergence of the two-level algorithm in the Hilbert space, i.e.,
$$|u-u^{(n+1)}|_1 \le C_2 |u-u^{(n)}|_1.$$

For the neural network iterates $\widehat{U}^{(n+1)}$, we can prove the convergence in the $L^2$-norm as in the one-level case:
\begin{theorem}\label{cor:conv:two-level:nn:uhat}
Under the Assumptions~\ref{assume1}, \ref{assume2}, \ref{assume3}, and \ref{assume4}, the iterates $\widehat{U}^{(n)}$ in \eqref{parallel:two-level:sol}
satisfy
$$\|u-\widehat{U}^{(n+1)}\|_0 \le \frac{C_p}{\tau} \left(|u-u^{(n+1)}|_1 + |u-u^{(n)}|_1+\frac{2}{1-C_2} (\epsilon_0+\epsilon) \right),$$
where $\epsilon$ is the local neural network solution error in Assumption~\ref{assume3},
$\epsilon_0$ is the coarse neural network solution error,
and the constant $C_2$ is the error reduction factor for the iterates
$u^{(n)}$ in the two-level additive Schwarz algorithm.
\end{theorem}

In our numerical results, we observed that the coarse problem in the above two-level method, Algorithm 2, does not help to speed up the iteration convergence, since the right hand side term $f+\triangle U^{(n)}$ in the coarse problem has
oscillatory and high contrast values near the overlapping region boundary,
see Fig.~\ref{fig:2d:smooth:coarse_forcing} in numerical results.
Such a coarse problem is hard to approximate with a good enough accuracy
using a coarse neural network function $U_0(x;\theta_0)$.

We can remove such a nonstandard coarse residual value problem in $f+\triangle U^{(n)}$
by proposing the following form of a partitioned neural network solution,
\begin{equation}\label{pu:two}
\widehat{U}(x;\theta)=\psi_0(x) U_0(x;\theta_0)+\sum_{i=1}^N (1-\psi_0(x))\phi_i(x)  U_i(x;\theta_i),
\end{equation}
where $0<\psi_0(x)<1$ and $\{\phi_i(x)\}_{i=1}^N$ form a partition of unity with each $\phi_i(x)$
supported in each subdomain $\Omega_i$, and $\theta=(\theta_0,\theta_1,\cdots,\theta_N)$.
Letting $\psi_i(x)=(1-\psi_0(x))\phi_i(x)$ for $i=1,\cdots,N$, we can also obtain $\{\psi_i(x) \}_{i=0}^N$ as
a partition of unity functions.
For the proposed partitioned neural network function in \eqref{pu:two},
we can apply the previous two-level iteration method to obtain the following algorithm:

{\bf Algorithm 3: Two-level method (partition of unity)} (input: $U^{(0)}$, output: $\widehat{U}^{(n+1)}$)

{\bf Step 0}: Let $U^{(0)}(x)$ be given and $n=0$.

{\bf Step 1}: For $i=0,\cdots,N$, find $\theta_i^{(n+1)}$ in $\psi_i(x)U_i(x;\theta_i^{(n+1)})$ for the following problem, note that $\Omega_0:=\Omega$,
\begin{equation*}
\begin{split}
-\triangle u&=f+\triangle (\sum_{j \ne i} \psi_j(x) U^{(n)}) \;\text{ in } \Omega_i, \\
u&=g-\sum_{j \ne i} \psi_j(x) U^{(n)} \; \text{ on } \partial \Omega_i \bigcap \partial \Omega.
\end{split}
\end{equation*}

%{\bf Step 1-2}: Find $\theta_0^{(n+1)}$ in $\psi_0(x) U_0(x;\theta_0^{(n+1)})$ for the coarse problem,
%\begin{equation*}\label{pb:coarse}
%\begin{split}
%-\triangle u_0^{(n+1)}&=f+\triangle(\sum_{j=1}^N \psi_j(x) U^{(n)}) \; \text{ in } \Omega,\\
%u_0&=0 \; \text{ on } \partial \Omega
%\end{split}
%\end{equation*}

{\bf Step 2}: Update $\triangle(\sum_{j \ne i} \psi_j(x) U^{(n+1)}(x))$ at each data set $X_{\Omega_i}$ and
$\sum_{j \ne i} \psi_j(x) U^{(n+1)}(x)$ at each data set $X_{\partial \Omega_i \bigcap \partial \Omega}$
using the iteration formula,
\begin{equation}\label{Un:pu}
U^{(n+1)}(x)=(1-\tau ) U^{(n)}(x) + \tau \widehat{U}^{(n+1)}(x),
\end{equation}
where
\begin{equation}\label{Uhat:pu}
\widehat{U}^{(n+1)}(x)=\sum_{i=0}^N \psi_i(x) U_i(x;\theta_i^{(n+1)}).
\end{equation}

{\bf Step 3}: Go to {\bf Step 1} with $n=n+1$ or set the output as $\widehat{U}^{(n+1)}$
if the stopping condition is met.

\vspace{0.3cm}

For the above Algorithm 3, we can consider the corresponding local solutions $\psi_i(x)u_i^{(n+1)}$ in the local Hilbert spaces $H^1(\Omega_i)$ and the coarse solution $\psi_0(x) u_0^{(n+1)}$ in the coarse Hilbert space $V_0$ of $H^1(\Omega)$. Using them, the iterates $u^{(n+1)}$ in the Hilbert space $H^1(\Omega)$ are obtained as
$$u^{(n+1)}=(1-\tau)u^{(n)}+\tau \sum_{i=0}^N \psi_i(x) u_i^{(n+1)},\quad \forall n\ge0.$$
With a given initial $u^{(0)}$ such that $u^{(0)}=g$ on $\partial \Omega$,
we can obtain the following error equation
$$u-u^{(n+1)}=u-u^{(n)}-\tau \sum_{i=0}^N \psi_i(x)(u_i^{(n+1)}-u^{(n)}).$$
By noticing that $\psi_i(x)(u_i^{(n+1)}-u^{(n)})$ is the Hilbert space solution to the following problem,
\begin{equation*}
\begin{split}
-\triangle(\psi_i(u_i^{(n+1)}-u^{(n)}))&=f+\triangle u^{(n)} \quad \text{in } \Omega_i\\
\psi_i(u_i^{(n+1)}-u^{(n)})&=g-u^{(n)} \quad \text{on } \partial \Omega_i \bigcap \partial \Omega,\\
\psi_i(u_i^{(n+1)}-u^{(n)})&=0 \quad \text{on } \partial \Omega_i \bigcap \Omega,
\end{split}
\end{equation*}
and using that $g-u^{(n)}=0$ on $\partial \Omega$, we can obtain $\psi_i(u_i^{(n+1)}-u^{(n)})=P_i(u-u^{(n)})$
and thus
$$u-u^{(n+1)}=u-u^{(n)}-\tau \sum_{i=0}^N P_i(u-u^{(n)})=(I-\tau \sum_{i=0}^N P_i) (u-u^{(n)}).$$
From the above error equation, for the iterates $U^{(n)}$ in Algorithm 3 we can obtain the same convergence result as in \eqref{two:conv:H1}
and for the iterates $\widehat{U}^{(n)}$ defined in \eqref{Uhat:pu} we can obtain a similar result by using the following identity:
$$u-\widehat{U}^{(n+1)}=\frac{1}{\tau} (u-U^{(n+1)})-(\frac{1}{\tau}-1) (u-U^{(n)}).$$
Differently from the previous Algorithms 1 and 2, we can also obtain
a stronger error estimate in the $H^1$-norm
for the iterates $\widehat{U}^{(n+1)}$ of the form in \eqref{Uhat:pu}.
In Algorithm 3, the right hand side term in the coarse problem, $f+\triangle(\sum_{i=1}^N \psi_j(x)U^{(n)})$,
presents smooth values over the domain $\Omega$ differently from the previous Algorithm 2, see Fig.~\ref{fig:2d:smooth:coarse_forcing} in numerical results.
Such a coarse problem can be well approximated by the coarse neural network with a small error $\epsilon_0$
and it thus helps to speed up the iteration convergence as increasing the number of subdomains, consistent
with our convergence result, see Theorem~\ref{cor:conv:two-level:nn:uhat} and the estimate in \eqref{two:conv:H1}. We note that a similar form of neural network approximation was also considered
in \cite{FBpinn,dolean2023multilevel}, where one-level and multi-level neural network functions are proposed
by using partition of unity functions and they are shown to be effective for modeling complex
and multi-frequency solutions.

\section{Numerical results}\label{sec:numerics}
In this section, we present numerical experiments of the proposed methods
for one-dimensional and two-dimensional model elliptic problems in an open interval $(a,b)$ and a unit rectangular domain, respectively.

In the one-dimensional examples, the domain is partitioned into $N$ uniform overlapping subintervals
and in the two-dimensional examples, into $N\times N$ uniform overlapping subrectangles.
These subintervals and subrectangles form subdomains in the partition.
In the partition, we set the overlapping width $\delta$ as one third of the subdomain diameter $H$, i.e., $H/\delta$ as $1/3$.
We note that in \cite{sineAct}
it was studied that the sine function preserves the derivative information
better than other commonly used activation functions.
In our numerical experiments, we thus employ fully connected neural network functions with the sine activation function to approximate the PDE solutions.

When training the parameters, we use the Adam optimizer~\cite{adam} with the learning rate as 0.001, and with the maximum number of epochs
%as 400,000 for training the parameters in the whole single domain problem and
as 10,000 for training the parameters in each local and coarse problems.
We set the maximum number of the outer iterations as 40 in the iterative methods and the initial value $U^{(0)}\equiv0$, unless otherwise mentioned.
In our numerical experiment, we also set $w_I=1$ and $W_B=500$ for the corresponding loss functions in each local and coarse problem, see \eqref{loss:two-terms}.

Our computation was performed on an Intel(R) Xeon(R) Gold 6248R CPU @ 3.00GHz and Quadro RTX 5000 and using the JAX library in Python. We used the jax.vmap to parallelize the local loss function computation in our Algorithms 1-3. In the Algorithms 2 and 3, at each outer iteration, we solved the local problems and the coarse problem sequentially. By solving them in fully parallel, we can also make the computation time in the Algorithms 2 and 3 faster. We also used the jax.jit to compile the local loss function calculation and data communication part. By compiling these operations, that are repeatedly used, we were able to reduce the overall training time further in our computation.
%In the following table, we list the hyper parameter settings for the neural network and training data set.
%**Include Table**

\subsection{One-dimensional examples}
For the one-dimensional case, we will consider two examples to show the performance of the proposed one-level and two-level iterative algorithms.
As a first example, we consider a smooth example and later we consider a multiscale example.

{\bf Smooth example.} We consider a model problem,
\begin{equation*}
\begin{split}
-u''(x)&=f(x),\quad \forall x \in (-1,\;1),\\
u(-1)&=u(1)=0,\\
\end{split}
\end{equation*}
where the right hand side function $f(x)$ is chosen to give the exact solution
\begin{equation}\label{model:1d:smooth}
u^*(x)=\sin(2 \pi x).
\end{equation}

We solve the above model problem using Algorithms 1-3, where local and coarse problems are solved using local neural networks and a coarse neural network at each outer iteration.
In the following Table~\ref{table:1d:smooth:info}, we list the hyper parameter settings for the neural network and training data set used in our computation for the above model problem in (\ref{model:1d:smooth}).
As increasing the number of subdomains $N$, we set the local network size smaller so that
the total number of parameters is similar for all the subdomain partition cases.
\begin{table}[ht!]
	\caption{
Hyper parameter settings for the computation results in Tables~\ref{table:1d:smooth:error} and \ref{table:1d:smooth:error:trained_U0}: The number of layers, neurons, parameters, and the number of interior and boundary training data points for the local and coarse networks (last row).}\label{table:1d:smooth:info}
	{\normalsize \renewcommand{\arraystretch}{1.0}
		\begin{center}
			\vskip-.3truecm
			\begin{tabular}{c|ccccc}
				
				\hline\hline
				& \multicolumn{5}{c}{Local problem} \\
				\hline\hline
				No. of subdomains & Layers & Neurons & Parameters  & Interior data & Boundary data \\
				\Xhline{3\arrayrulewidth}
				$5$               & 4 & 7 & 190 & 198 & 2\\
				$10$              & 4 & 5 & 106 & 98 & 2\\
				$20$              & 4 & 3 & 46 & 48  & 2\\
				\Xhline{0.8\arrayrulewidth}
				Coarse problem    & 4 & 5 & 106 & 98  & 2 \\
				\Xhline{3\arrayrulewidth}
			\end{tabular}
		\end{center}
	}
	\vskip-.2truecm
\end{table}

In Table~\ref{table:1d:smooth:error}, the mean values of relative $L^2$-errors to the exact solution are reported for the neural network solutions obtained from Algorithms 1-3 using five different seeds and with the zero initial $U^{(0)}$.
The error results show that the convergence in the neural network solution gets
slower in Algorithm 1 as increasing the number of subdomains $N$, that is consistent with
our convergence analysis in the previous section.
On the other hand, the convergence of the neural network solution in Algorithm 2
shows a similar behavior as in Algorithm 1, with  even larger errors than
those in Algorithm 1, contrary to our theory developed for the two-level algorithm.
In Algorithm 3, i.e., the two-level algorithm for the partition of unity neural networks,
the error results clearly show that its convergence rate is robust to the number of subdomains,
consistent with our theory for the two-level algorithm.

In Fig.~\ref{fig:1d:smooth}, the error decay history in Algorithms 1-3 is also presented
and compared over the outer iterations for 5, 10, and 20 subdomain partitions.
For all the subdomain partition cases, Algorithm 3 shows similar error decay plots,
while Algorithms 1 and 2 show slower error decay rates as increasing the number of subdomains in
the partition. Algorithm 2 shows even slower decay rates than those in Algorithm 1.

\begin{table}[ht!]
	\caption{Smooth example in (\ref{model:1d:smooth}):
The mean values of relative $L^2$-errors in Algorithms 1-3
trained with five different seeds, as increasing the number of subdomains $N$.}\label{table:1d:smooth:error}
	{\normalsize \renewcommand{\arraystretch}{1.0}
		\begin{center}
			\vskip-.3truecm
			\begin{tabular}{cccc}

				\hline\hline
				No. of subdomains & Algorithm 1  & Algorithm 2 & Algorithm 3 \\
				\Xhline{3\arrayrulewidth}
				5  & 0.0010 & 0.0013 & 0.0021\\
				10 & 0.0018 & 0.0054 & 0.0022\\
				20 & 0.1172 & 0.2364 & 0.0046\\
				%\Xhline{0.8\arrayrulewidth}
				%Single-domain  & 0.0013 & 0.0013 & 0.0013\\
				\Xhline{3\arrayrulewidth}
			\end{tabular}
		\end{center}
	}
	\vskip-.2truecm
\end{table}

\begin{figure}[ht!]
	\centering
	\begin{subfigure}{0.32\textwidth}
		\includegraphics[width=\textwidth]{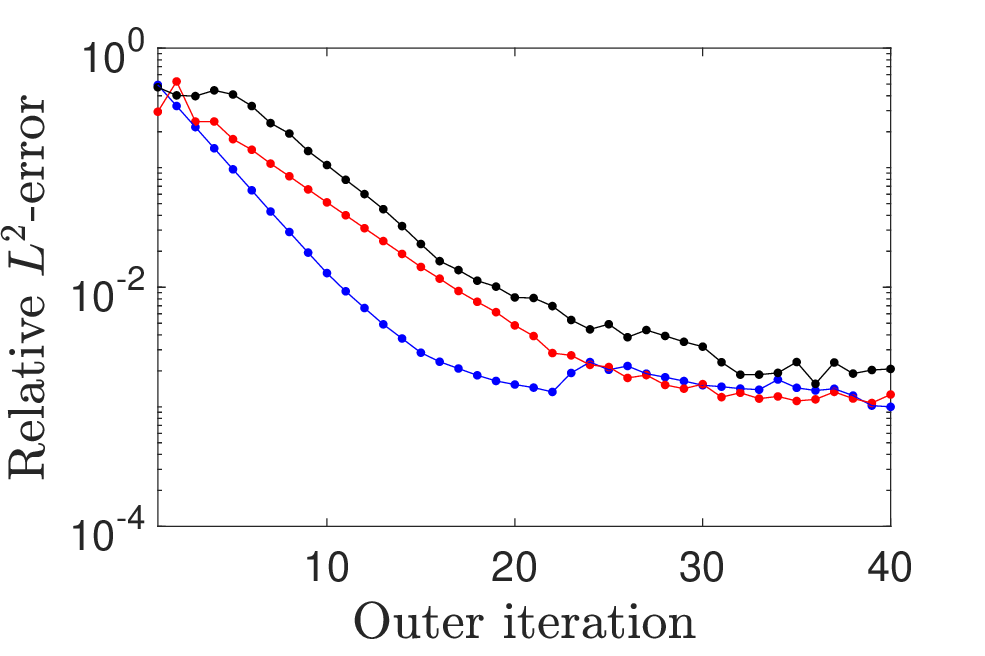}
		\caption{5 subdomains}
		\label{fig:1d:smooth:first}
	\end{subfigure}
	\hfill
	\begin{subfigure}{0.32\textwidth}
		\includegraphics[width=\textwidth]{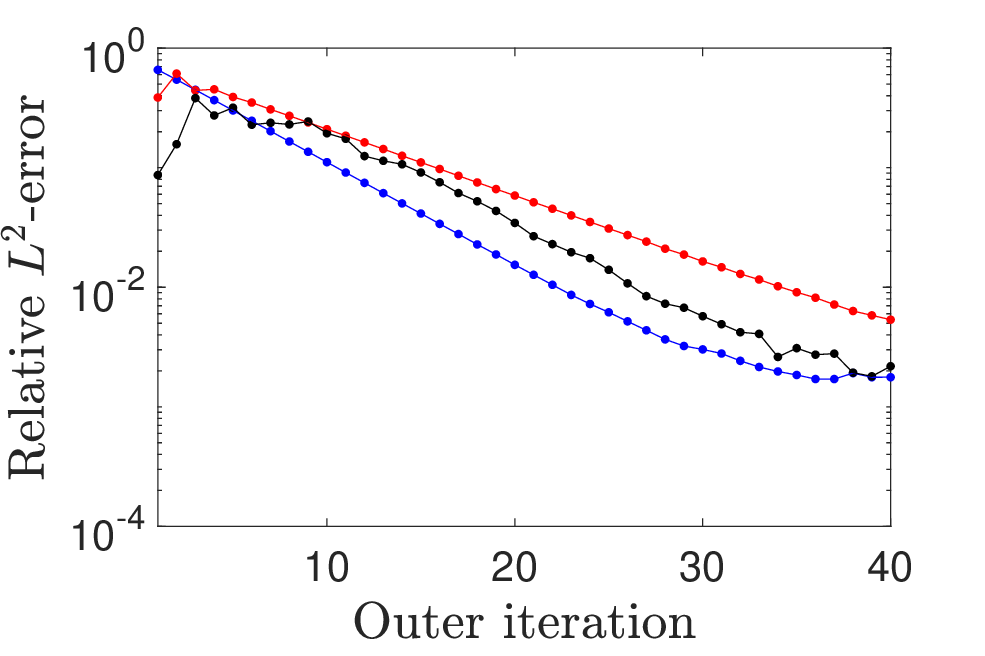}
		\caption{10 subdomains}
		\label{fig:1d:smooth:second}
	\end{subfigure}
	\hfill
	\begin{subfigure}{0.32\textwidth}
		\includegraphics[width=\textwidth]{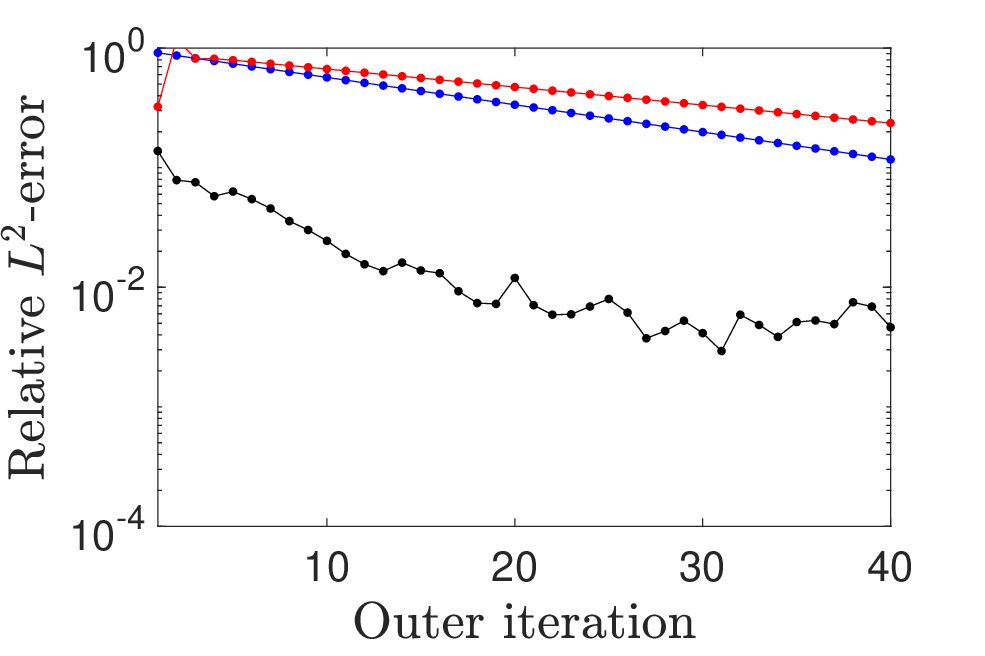}
		\caption{20 subdomains}
		\label{fig:1d:smooth:third}
	\end{subfigure}
	
	\caption{Smooth example in (\ref{model:1d:smooth}): Error decay history in Algorithm 1 (blue),
Algorithm 2 (red), and Algorithm 3 (black).
%(a) $N=5$ subdomains, (b) $N=10$ subdomains, (c) $N=20$ subdomains.
}
	\label{fig:1d:smooth}
\end{figure}

The training performance in our iteration methods can depend on the choice of initial value $U^{(0)}$.
For the same smooth example, we employ an initial value $U^{(0)}$ as a trained solution with
a relatively small neural network, training data, and training epochs.
For the network used for initial value setting, we form the fully connected neural network with 106 parameters and 100 training data points selected randomly from the interval, which include 98 interior data points from the interval and the two end points.
We also train the initial network
using the Adam optimizer with the learning rate of 0.001 and with the 10,000 training epochs.
For Algorithms 1-3, we use the trained solution to set the initial value $U^{(0)}$ and report
the relative $L^2$-errors for the neural network approximation, obtained from  five different seeds.
The error results are listed in Table~\ref{table:1d:smooth:error:trained_U0}.
With a good initial value, closer to the solution, we observe that
the errors are smaller than those obtained in Table~\ref{table:1d:smooth:error}.
Similarly as before, in Algorithms 1 and 2, the solution convergence gets slower
as more subdomains in the partition.
In Algorithm 3, we obtained error results that are more robust to the number of subdomains
than those in Algorithms 1 and 2.
We note that in Algorithm 3 we observed that
some trained solutions starting from the zero initial $U^{(0)}$
fall into a local minimum at early outer iterations and
the use of a trained initial $U^{(0)}$ can help to avoid such wrong trained solutions.
In Algorithm 1, for all choices of the initial value $U^{(0)}$, we obtained
successful training results without such a local minimum problem.

In Fig.~\ref{fig:1d:smooth:trained_U0}, we plot the error decay history over the outer iterations
in Algorithms 1-3 with a trained initial $U^{(0)}$.
As increasing the number of subdomains in the partition, in Algorithms 1 and 2, the error decay rates
are getting slower while in Algorithm 3 the error decay rates are robust to the number of subdomains.
Starting with a good initial $U^{(0)}$, we also observe that Algorithm 3 performs
more efficiently than Algorithms 1 and 2 for all the subdomain partition cases, compared
to the error decay plots in Fig.~\ref{fig:1d:smooth} with the zero initial $U^{(0)}$.

\begin{table}[ht!]
	\caption{Smooth example in (\ref{model:1d:smooth}):
The mean values of relative $L^2$-errors in Algorithms 1-3
trained with five different seeds, as increasing the number of subdomains $N$ and with a trained initial $U^{(0)}$.}\label{table:1d:smooth:error:trained_U0}
	{\normalsize \renewcommand{\arraystretch}{1.0}
		\begin{center}
			\vskip-.3truecm
			\begin{tabular}{cccc}
				
				\hline\hline
				No. of subdomains & Algorithm 1  & Algorithm 2 & Algorithm 3 \\
				\Xhline{3\arrayrulewidth}
				Initial $U^{(0)}$ & 0.0282 & 0.0282 & 0.0282 \\
				\Xhline{0.8\arrayrulewidth}
				5  & 0.0022 & 0.0129 & 0.0012 \\
				10 & 0.0142 & 0.0315 & 0.0030 \\
				20 & 0.0258 & 0.1840 & 0.0037 \\
				%\Xhline{0.8\arrayrulewidth}
				%Single-domain  & 0.0013 & 0.0013 & 0.0013\\
				\Xhline{3\arrayrulewidth}
			\end{tabular}
		\end{center}
	}
	\vskip-.2truecm
\end{table}

\begin{figure}[ht!]
	\centering
	\begin{subfigure}{0.32\textwidth}
		\includegraphics[width=\textwidth]{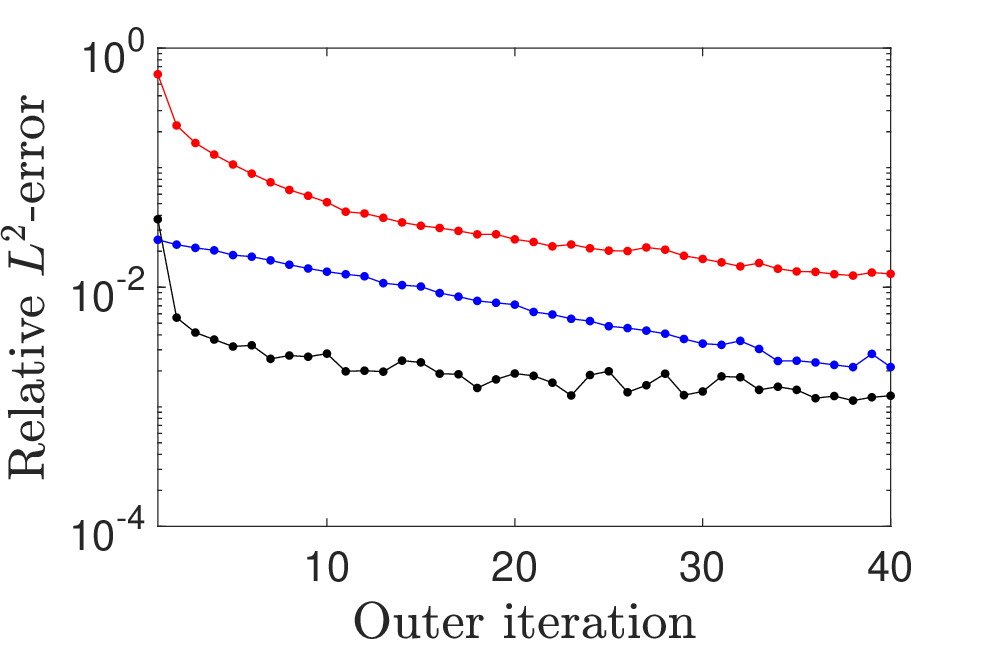}
		\caption{5 subdomains}
		\label{fig:1d:smooth:trained_U0:first}
	\end{subfigure}
	\hfill
	\begin{subfigure}{0.32\textwidth}
		\includegraphics[width=\textwidth]{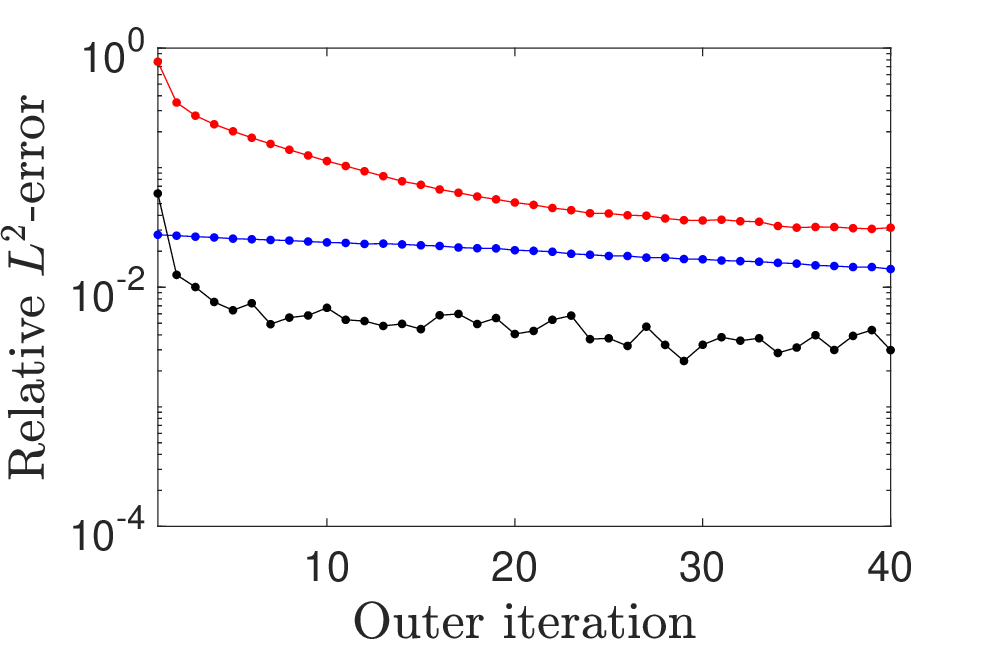}
		\caption{10 subdomains}
		\label{fig:1d:smooth:trained_U0:second}
	\end{subfigure}
	\hfill
	\begin{subfigure}{0.32\textwidth}
		\includegraphics[width=\textwidth]{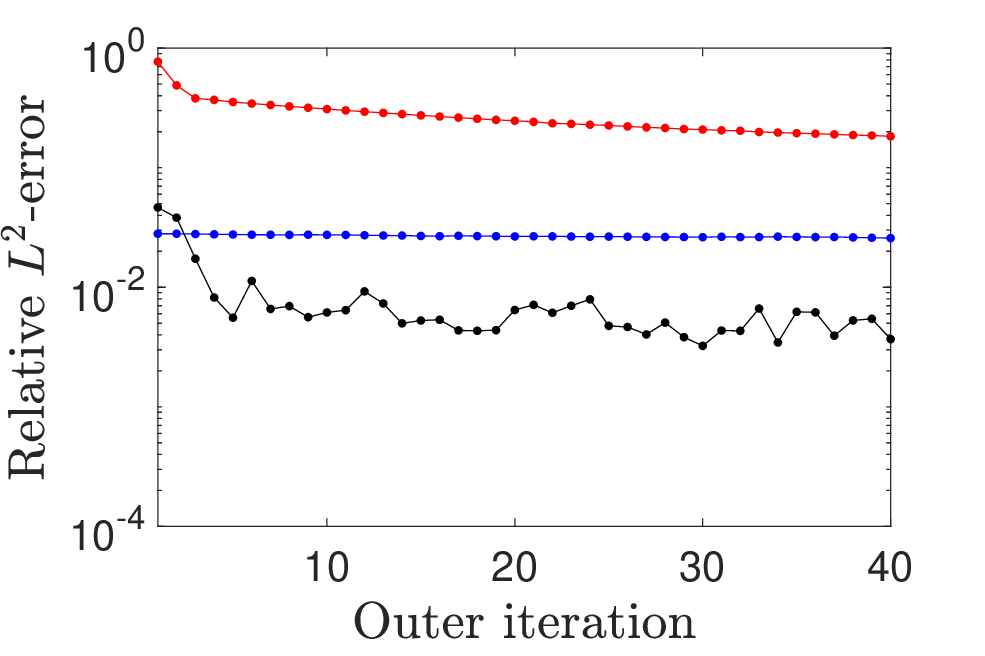}
		\caption{20 subdomains}
		\label{fig:1d:smooth:trained_U0:third}
	\end{subfigure}
	
	\caption{Smooth example in (\ref{model:1d:smooth}): Error decay history with a trained $U^{(0)}$ in Algorithm 1 (blue), Algorithm 2 (red), and Algorithm 3 (black).
% (a) $N=5$ subdomains, (b) $N=10$ subdomains, (c) $N=20$ subdomains.
 }
	\label{fig:1d:smooth:trained_U0}
\end{figure}

{\bf Multiscale example.}
We now perform the proposed methods for a more challenging multiscale model problem with its exact solution given as
\begin{equation}\label{model:1d:multiscale}
	u^*(x)=2\sin(4\pi x)+\sin(8\pi x)+\frac{1}{2}\sin(16\pi x),\quad x \in (-1,\; 1).
\end{equation}
The plots of the solution and the corresponding forcing term $f(x)$ in the differential equation, $-u_{xx}(x)=f(x)$,
are presented in Fig.~\ref{fig:1d:multiscale:exact}.
We note that neural network approximation to such a multiscale feature solution is known to be
difficult task and several successful approaches have been also studied in earlier works~\cite{cai2020phase,liu2020multi} by using
different approaches from ours.

\begin{figure}[ht!]
	\centering
	\includegraphics[width=0.4\textwidth]{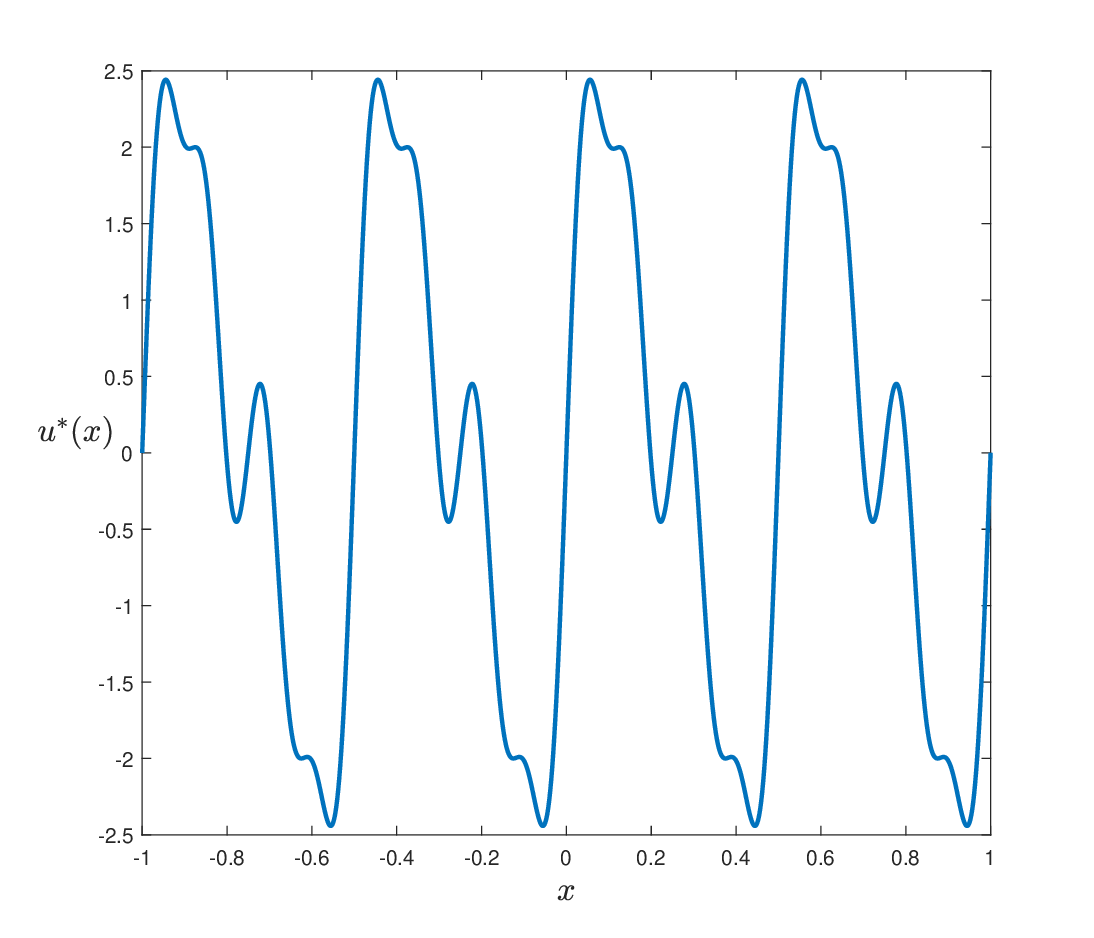}
	\includegraphics[width=0.4\textwidth]{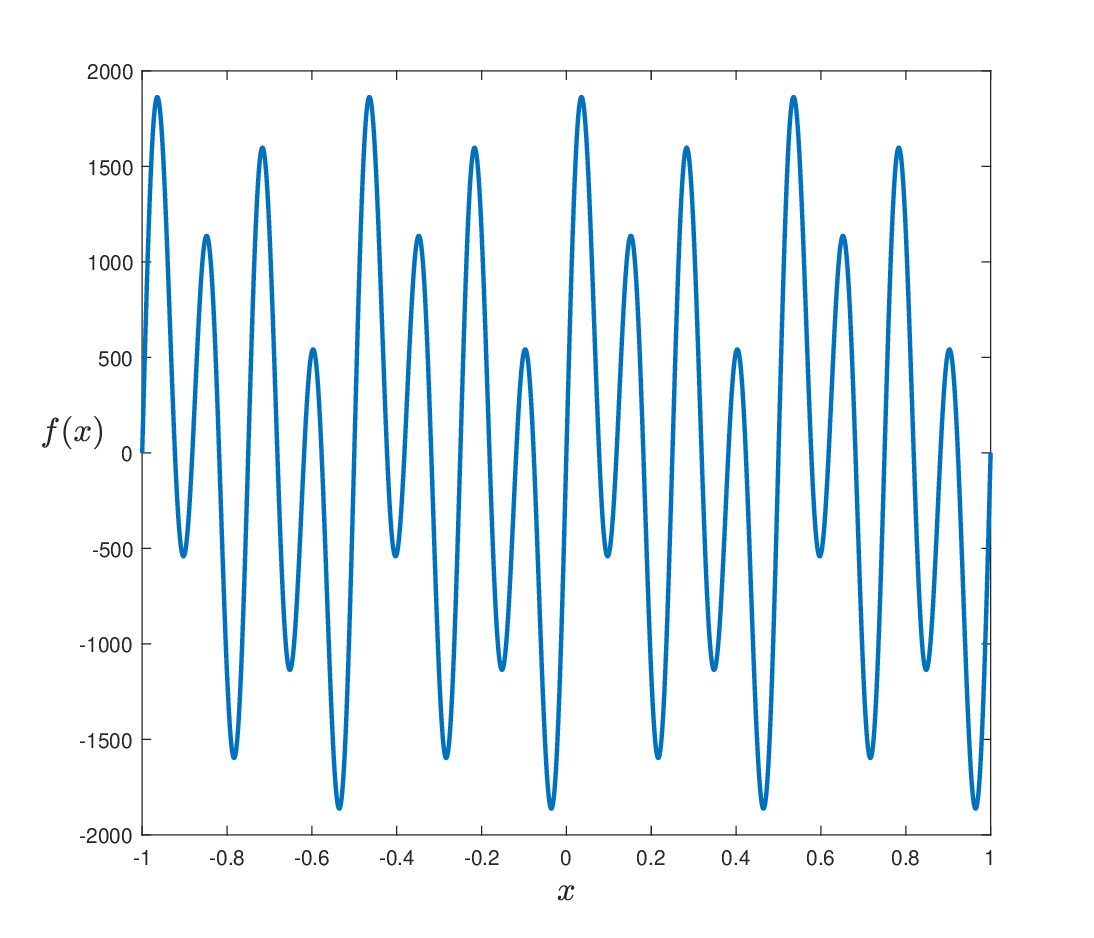}
	\vskip-.4truecm
	\caption{Multiscale example in~\eqref{model:1d:multiscale}: Exact solution $u^*(x)$ (left) and forcing term $f(x)$ (right)}
	\label{fig:1d:multiscale:exact}
\end{figure}

In Table~\ref{table:1d:multiscale:info}, we list the hyper parameter settings for the neural network and training data set used in our computation for the above multiscale model problem in~\eqref{model:1d:multiscale}.
Similarly as in the smooth example case, we employed smaller local neural networks as
increasing the number of subdomains.

\begin{table}[ht!]
\caption{Hyper parameter settings for the computation results in Tables~\ref{table:1d:multiscale:error}
and \ref{table:1d:multiscale:error:trained_U0}: The number of layers, neurons, parameters, and the number of interior and boundary training data points for the local and coarse networks (last row).}\label{table:1d:multiscale:info}
	{\normalsize \renewcommand{\arraystretch}{1.0}
		\begin{center}
			\vskip-.3truecm
			\begin{tabular}{c|ccccc}
				
				\hline\hline
				& \multicolumn{5}{c}{Local problem} \\
				\hline\hline
				No. of subdomains & Layers & Neurons & Parameters  & Interior data & Boundary data \\
				\Xhline{3\arrayrulewidth}
				$4$               & 2 & 18 & 397 & 398 & 2\\
				$8$               & 2 & 12 & 193 & 198 & 2\\
				$16$              & 2 & 8 & 97 & 98  & 2\\
				\Xhline{0.8\arrayrulewidth}
				%Single-domain     & 4 & 22 & 1585 & 1598 \\
                Coarse problem     & 2 & 18 & 397 & 398 & 2 \\
				\Xhline{3\arrayrulewidth}
			\end{tabular}
		\end{center}
	}
	\vskip-.2truecm
\end{table}

In Table~\ref{table:1d:multiscale:error}, we report the error results obtained from Algorithms 1-3
for the multiscale example in \eqref{model:1d:multiscale}.
In this multiscale example, we needed more outer iterations than in the previous smooth example
to obtain accurate enough approximate solutions
and we thus performed 100 outer iterations in our computation.
As increasing the number of subdomains $N$, the convergence of the neural network approximation
is getting slower in Algorithms 1 and 2.
In Algorithm 3, we even observe smaller errors with the 8 and 16 subdomain cases than with the 4 subdomain case,
in contrast to the error results for the smooth example in Tables~\ref{table:1d:smooth:error} and \ref{table:1d:smooth:error:trained_U0}.
The two-level algorithm, i.e., Algorithm 3, for the partitioned neural networks built on the partition of unity functions, performs more effectively for the multiscale example than for the smooth example as
more subdomains are introduced in the partition.
In Fig.~\ref{fig:1d:multiscale}, the error decay history for the approximate solutions in Table~\ref{table:1d:multiscale:error}
is reported over the 100 outer iterations.
As increasing the number of subdomains, the error decay in Algorithms 1 and 2 is getting slower
while those in Algorithm 3 are observed to be robust.
In addition, for all the subdomain partition cases, Algorithm 3 is
observed to give faster convergent approximate solutions than those in Algorithms 1 and 2.

%%%%%%%%%%%%%%%%%%%%%%%%%%%%%%%%%%%%%%%%%%%%%%%%%%%%%
%%%%%%%%%%%%%%%%%%%%%% 1D multiscale initial network

%%%%%%%%%%%%%%%%%%%%%%%%%%%%%%%%%%%%%%%%%%%%%%%%%%%%%

\begin{table}[ht!]
	\caption{Multiscale example in (\ref{model:1d:multiscale}): The mean values of relative $L^2$-errors in Algorithms 1-3 trained with five different seeds, as increasing the number of subdomains $N$.}\label{table:1d:multiscale:error}
	{\normalsize \renewcommand{\arraystretch}{1.0}
		\begin{center}
			\vskip-.3truecm
			\begin{tabular}{cccc}
				
				\hline\hline
				No. of subdomains & Algorithm 1  & Algorithm 2 & Algorithm 3 \\
				\Xhline{3\arrayrulewidth}
				4      & 0.0062 & 0.0136 & 0.0042 \\
				8      & 0.0088 & 0.0641 & 0.0028 \\
				16     & 0.0058 & 0.4526 & 0.0028 \\
				%\Xhline{0.8\arrayrulewidth}
				%Single-domain & 0.0024 & 0.0024 & 0.0024 \\
				\Xhline{3\arrayrulewidth}
			\end{tabular}
		\end{center}
	}
	\vskip-.2truecm
\end{table}

\begin{figure}[ht!]
	\centering
	\begin{subfigure}{0.32\textwidth}
		\includegraphics[width=\textwidth]{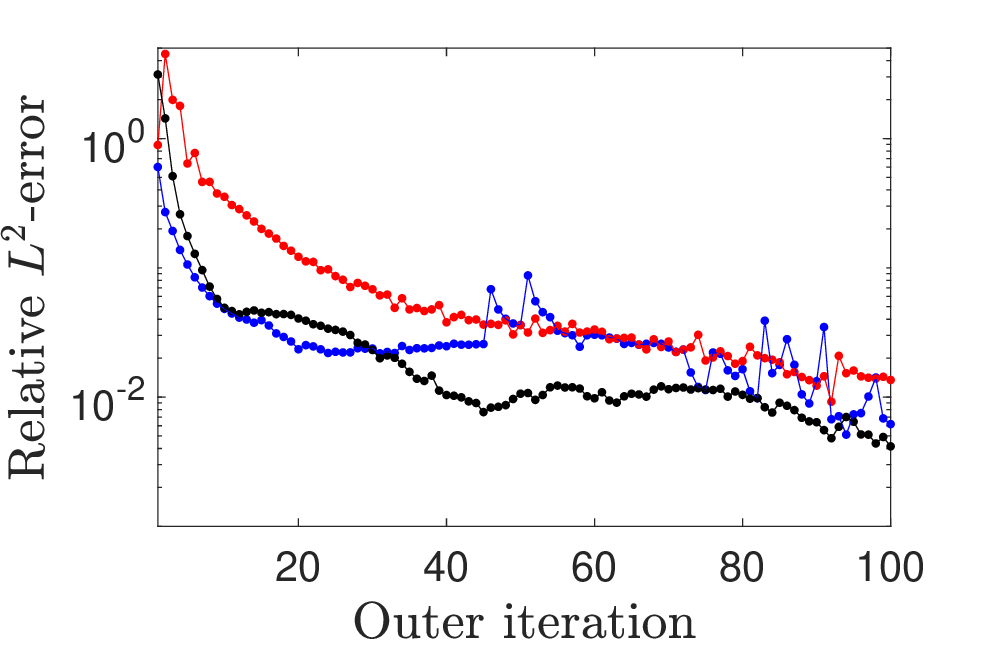}
		\caption{4 subdomains}
		\label{fig:1d:multiscale:first}
	\end{subfigure}
	\hfill
	\begin{subfigure}{0.32\textwidth}
		\includegraphics[width=\textwidth]{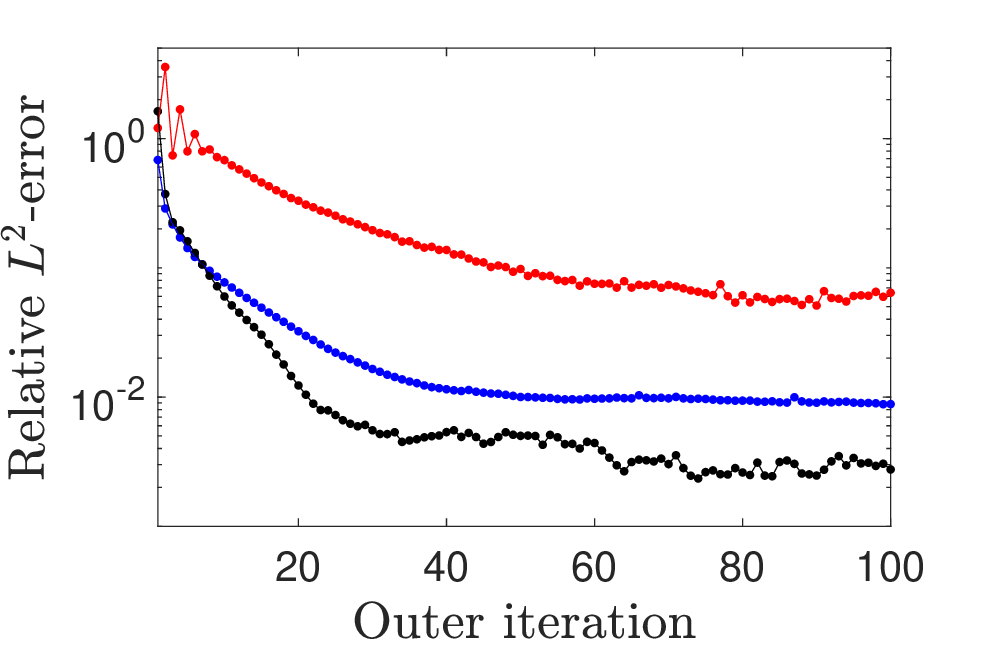}
		\caption{8 subdomains}
		\label{fig:1d:multiscale:second}
	\end{subfigure}
	\hfill
	\begin{subfigure}{0.32\textwidth}
		\includegraphics[width=\textwidth]{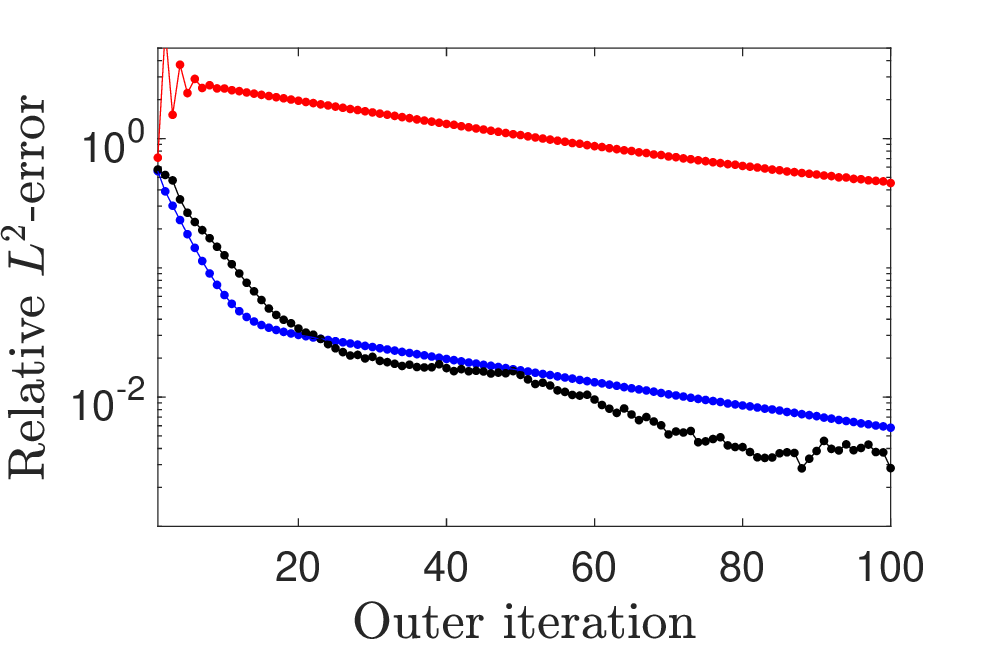}
		\caption{16 subdomains}
		\label{fig:1d:multiscale:third}
	\end{subfigure}
	
	\caption{Multiscale example in (\ref{model:1d:multiscale}): Error decay history in Algorithm 1 (blue),
		Algorithm 2 (red), and Algorithm 3 (black).}
	\label{fig:1d:multiscale}
\end{figure}

%%%%%%%%%%%%%%%%%%%%%%%%%%%%%%%%%%%%%%%%%%%%%%%%%%%%%%%%%%%%%%%%%%%%%%%%%%%%%%%%%%%%%%%%%%%%%%%%%%%%%%%%%%%%%%%%%%%%%%%%%%%%%%%%%%%%%%%%%%
%%%%%%%%%%%%%%%%%%%%%%%%%%%%%%%%%%%%%%%%%%%%%%%%%%%%% 1D multiscale with trained U0

We now employ an initial value $U^{(0)}$ as a trained solution with
a relatively small neural network, training data, and training epochs.
For the initial value setting, we form a fully connected neural network with 397 parameters and 400 training data points selected randomly from the interval, which include 398 interior data points from the interval and the two end points. The learning rate and the training epochs for $U^{(0)}$ are set to the same as in the smooth example.
The error results are reported in Table~\ref{table:1d:multiscale:error:trained_U0} for the trained
$U^{(0)}$ case. In Algorithms 1 and 2, similar convergence results as in the previous Table~\ref{table:1d:multiscale:error} are obtained and in Algorithm 3, the neural network solution convergence is again robust
to the number of subdomains. For the obtained results in Table~\ref{table:1d:multiscale:error:trained_U0}, the error decay history up to 100 outer iterations
is presented in Fig.~\ref{fig:1d:multiscale:trained_U0}.
It shows scalable convergence for Algorithm 3, while the slower convergence in Algorithm 1 as the more
subdomains in the partition.

\begin{table}[ht!]
	\caption{Multiscale example in (\ref{model:1d:multiscale}): The mean values of relative $L^2$-errors in Algorithms 1-3 trained with five different seeds, as increasing the number of subdomains $N$ and with a trained initial $U^{(0)}$.}\label{table:1d:multiscale:error:trained_U0}
	{\normalsize \renewcommand{\arraystretch}{1.0}
		\begin{center}
			\vskip-.3truecm
			\begin{tabular}{cccc}
				
				\hline\hline
				No. of subdomains & Algorithm 1  & Algorithm 2 & Algorithm 3 \\
				\Xhline{3\arrayrulewidth}
				Initial $U^{(0)}$ & 0.0397 & 0.0397 & 0.0397\\
				\Xhline{0.8\arrayrulewidth}
				4                 & 0.0026 & 0.0137 & 0.0043 \\
				8                 & 0.0047 & 0.0941 & 0.0058 \\
				16                & 0.0177 & 0.0585 & 0.0041 \\
				%\Xhline{0.8\arrayrulewidth}
				%Single-domain & 0.0024 & 0.0024 & 0.0024 \\
				\Xhline{3\arrayrulewidth}
			\end{tabular}
		\end{center}
	}
	\vskip-.2truecm
\end{table}

\begin{figure}[ht!]
	\centering
	\begin{subfigure}{0.32\textwidth}
		\includegraphics[width=\textwidth]{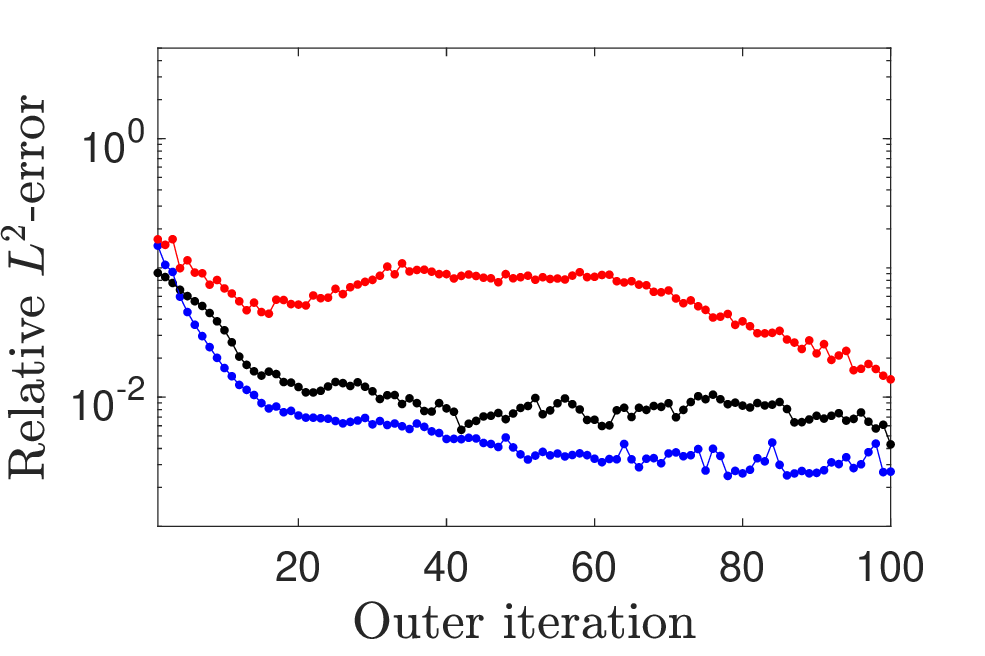}
		\caption{4 subdomains}
		\label{fig:1d:multiscale:trained_U0:first}
	\end{subfigure}
	\hfill
	\begin{subfigure}{0.32\textwidth}
		\includegraphics[width=\textwidth]{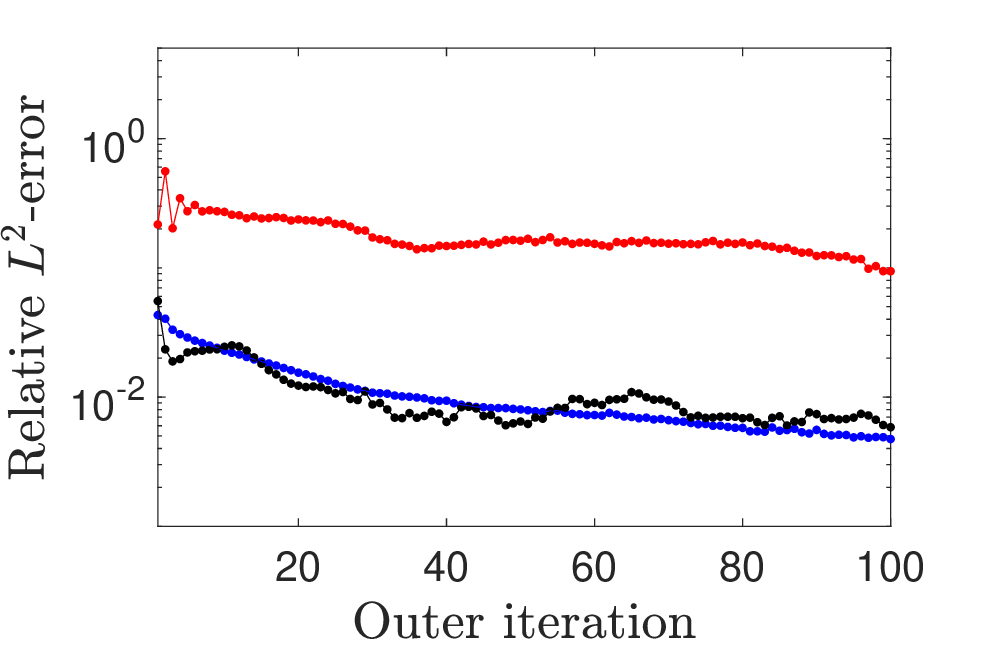}
		\caption{8 subdomains}
		\label{fig:1d:multiscale:trained_U0:second}
	\end{subfigure}
	\hfill
	\begin{subfigure}{0.32\textwidth}
		\includegraphics[width=\textwidth]{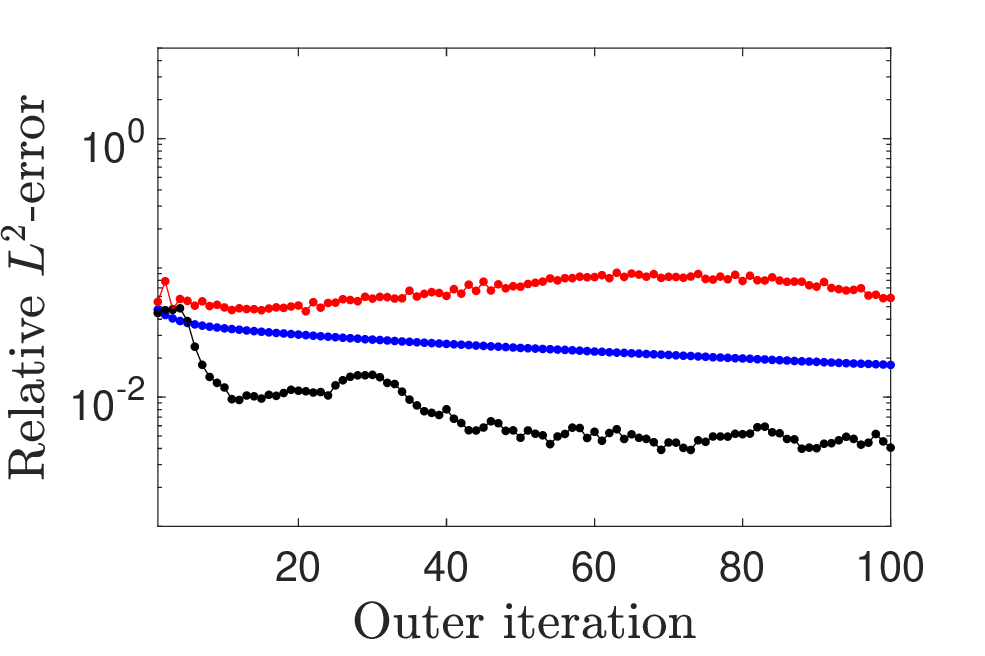}
		\caption{16 subdomains}
		\label{fig:1d:multiscale:trained_U0:third}
	\end{subfigure}
	
	\caption{Multiscale example in (\ref{model:1d:multiscale}): Error decay history with a trained $U^{(0)}$ in Algorithm 1 (blue),
		Algorithm 2 (red), and Algorithm 3 (black).}
	\label{fig:1d:multiscale:trained_U0}
\end{figure}

%%%%%%%%%%%%%%%%%%%%%%%%%%%%%%%%%%%%%%%%%%%%%%%%%%%%%%%%%%%%%%%%%%%%%%%%%%%%%%%%%%%%%%%%%%%%%%%%%%%%%%%%%%%%%%%%%%%%%%%%%%%%%%%%%%%%%%%%%%
%%%%%%%%%%%%%%%%%%%%%%%%%%%%%%%%%%%%%%%%%%%%%%%%%%%%%%%%%%%%%%%%%%%%%%%%%%%%%%%%%%%%%%%%%%%%%%%%%%%%%%%%%%%%%%%%%%%%%%%%%%%%%%%%%%%%%%%%%%
%{\bf Large domain example.}

\subsection{Two-dimensional examples}
For the two-dimensional case, we will test our proposed methods
for an oscillatory example defined in a large problem domain, a smooth example, and a multiscale example.

{\bf{Oscillatory example in a large domain.}} When applied to approximate a model problem solution defined in a large problem domain, a single PINN approach often suffers from a longer parameter training time.
Such a long training time problem is caused by the need for more data and more network parameters as the problem domain size becomes larger. To resolve this issue, a domain scaling approach can be used by transforming the large problem domain into a smaller domain and solving the smaller domain problem by using neural network approximation.
However, such a domain scaling approach transforms the original solution into a much more oscillatory solution, causing difficulties in the neural network approximation for the smaller domain problem.
To show the effectiveness of the partitioned neural network approach for such a large domain problem,
we consider a Poisson problem in a large problem domain with the following oscillatory solution,
\begin{align}\label{model:2d:largedomain}
	u^*(x,y)=\text{cos}(\pi x)\text{cos}(\pi y),\quad (x,y)\in(0,100)\times(0,4).
\end{align}
The exact solution plot is presented in Fig.~\ref{fig:2d:large:exact}.
We can easily see that
the use of domain scaling approach, for example, scaling the domain into a unit square domain,
transforms the original model problem solution into a much more oscillatory solution
in the unit square domain,
that is still much harder to be resolved by the single neural network approximation.

\begin{figure}[ht!]
	\centering
	\includegraphics[width=0.8\textwidth]{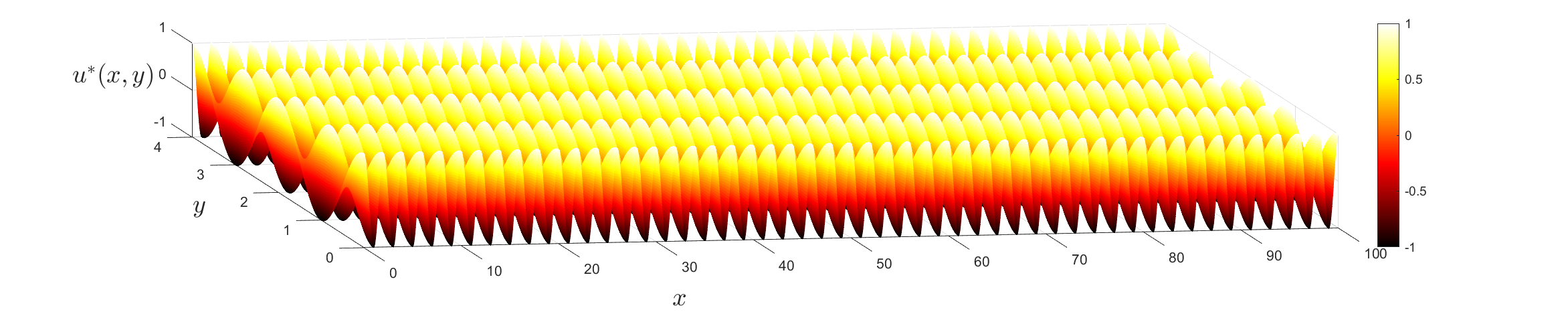} \\
	\includegraphics[width=0.8\textwidth]{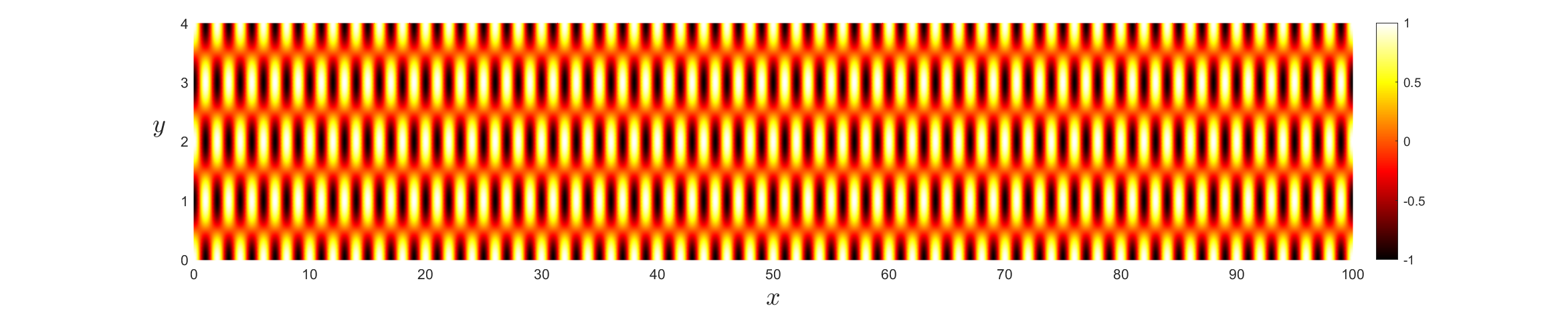}
	\vskip-.4truecm
	\caption{Oscillatory example in (\ref{model:2d:largedomain}): Surface plot (top) and contour plot (bottom) of the exact solution $u^*(x,y)$ in a large problem domain $(0,\,100)\times(0,\;4)$.}
	\label{fig:2d:large:exact}
\end{figure}

We will solve the above model problem in \eqref{model:2d:largedomain} by using single neural network approximation
and by using partitioned neural network approximation.
In our computation, we simply use Algorithm 1 for the partitioned neural network approximation
and partition the problem domain into $50\times4$ subdomains.
In Table~\ref{table:2d:large:info}, we list the number of parameters, and interior and boundary training data points, used in our computation.
In the single neural network case, we perform our computation by employing three kinds of networks,
i.e., increasing the number of parameters and the associated training data points.
We also note that the total number of parameters in the partitioned neural networks is comparable to the $40481$ parameter case in the single neural network approximation.
\begin{table}[ht!]
\caption{Hyper parameter settings for the computation results in Table~\ref{table:2d:large:error}: The number of layers, neurons, parameters, and the number of interior and boundary training data points for each local network (second row)
in the partitioned neural networks and for the single network (last three rows).}\label{table:2d:large:info}
	{\normalsize \renewcommand{\arraystretch}{1.0}
		\begin{center}
			\vskip-.3truecm
			\begin{tabular}{cccccc}

				\hline\hline
				No. of subdomains & Layers & Neurons  & Parameters & Interior data & Boundary data \\
				\Xhline{3\arrayrulewidth}
				$50\times 4$      &   4    &   7    & 197     &  250     &  50     \\
				\Xhline{0.8\arrayrulewidth}
				Single domain    &   4    &   55   & 9461    &  10000   &  2000    \\
				                  &   4    &   115  & 40481   &  50000   &  10000   \\
				                  &   5    &   150  & 91201   &  100000  &  20000   \\
				\Xhline{3\arrayrulewidth}
			\end{tabular}
		\end{center}
	}
	\vskip-.2truecm
\end{table}

In Table~\ref{table:2d:large:error},
the relative $L^2$-errors and the total training time in GPU are reported for the large domain oscillatory example in
\eqref{model:2d:largedomain}.
The error results clearly show that the partitioned neural network approximation
gives much more accurate solutions than those in the single neural network approximation.
Even with a much larger number of parameters in the single neural network approximation,
the error result only gives about six percent relative errors at the best case.
The timing result in the partitioned neural network approximation
is also much less than those in the single domain case.
Based on these error and timing results for this particular example,
we can see that the use of partitioned neural networks has
advantages over the single neural network approximation
in terms of the solution accuracy and training efficiency.
In our case, we trained the parameters in the partitioned neural networks
by using iteration methods, Algorithm 1, and with that we perform
the local parameter training procedure in parallel inside every outer iteration.
The communication between neighboring neural network solutions occurs once
before starting next outer iteration in contrast to
the previously developed partitioned neural network approximation methods~\cite{Cpinn,Xpinn,FBpinn},
where the communication is needed every training epoch.

\begin{table}[ht!]
	\caption{Large domain oscillatory example in (\ref{model:2d:largedomain}): Relative $L^2$-errors and the training time in the partitioned neural network approximation ($50\times4$ and Algorithm 1) and
in the single neural network approximation (Single-domain), and Para denotes the number of parameters in the local neural network ($N=50\times 4$) or the number of parameters in the single neural network (Single-domain).}\label{table:2d:large:error}
	{\normalsize \renewcommand{\arraystretch}{1.0}
		\begin{center}
			\vskip-.3truecm
			\begin{tabular}{cccc}

				\hline\hline
				No. of subdomains & Para & $L^2$-error & Time (sec)   \\
				\Xhline{3\arrayrulewidth}
				$50\times 4$      & 197    & 0.0091               & 2856.1    \\
				\Xhline{0.8\arrayrulewidth}
				Single-domain     & 9461   & 0.3347               & 4670.3  \\
				                  & 40481  & 0.2464               & 26260.8   \\
				                  & 91201  & 0.0650               & 104961.5  \\
				\Xhline{3\arrayrulewidth}
			\end{tabular}
		\end{center}
	}
	\vskip-.2truecm
\end{table}

In Fig.~\ref{fig:2d:large}, we present the error decay history of the computational results in Table~\ref{table:2d:large:error}.
The partitioned neural network case shows decreasing errors over the outer iterations.
On the other hand, the errors in the single neural network approximation fluctuate over
the training epochs. We note that in the plots the case of $10$ outer iterations corresponds to the case of 10 times 10000 training epochs in the single neural network approximation.
The error decay plots also support the use of the partitioned neural networks
and the iteration methods
for the training performance and efficiency enhancement in the neural network approximation.
%In our next numerical examples, we will compare the performance of the iteration methods, Algorithms 1-3,
%for the partitioned neural network approximation.

\begin{figure}[ht!]
	\begin{center}
		\includegraphics[width=0.7\textwidth]{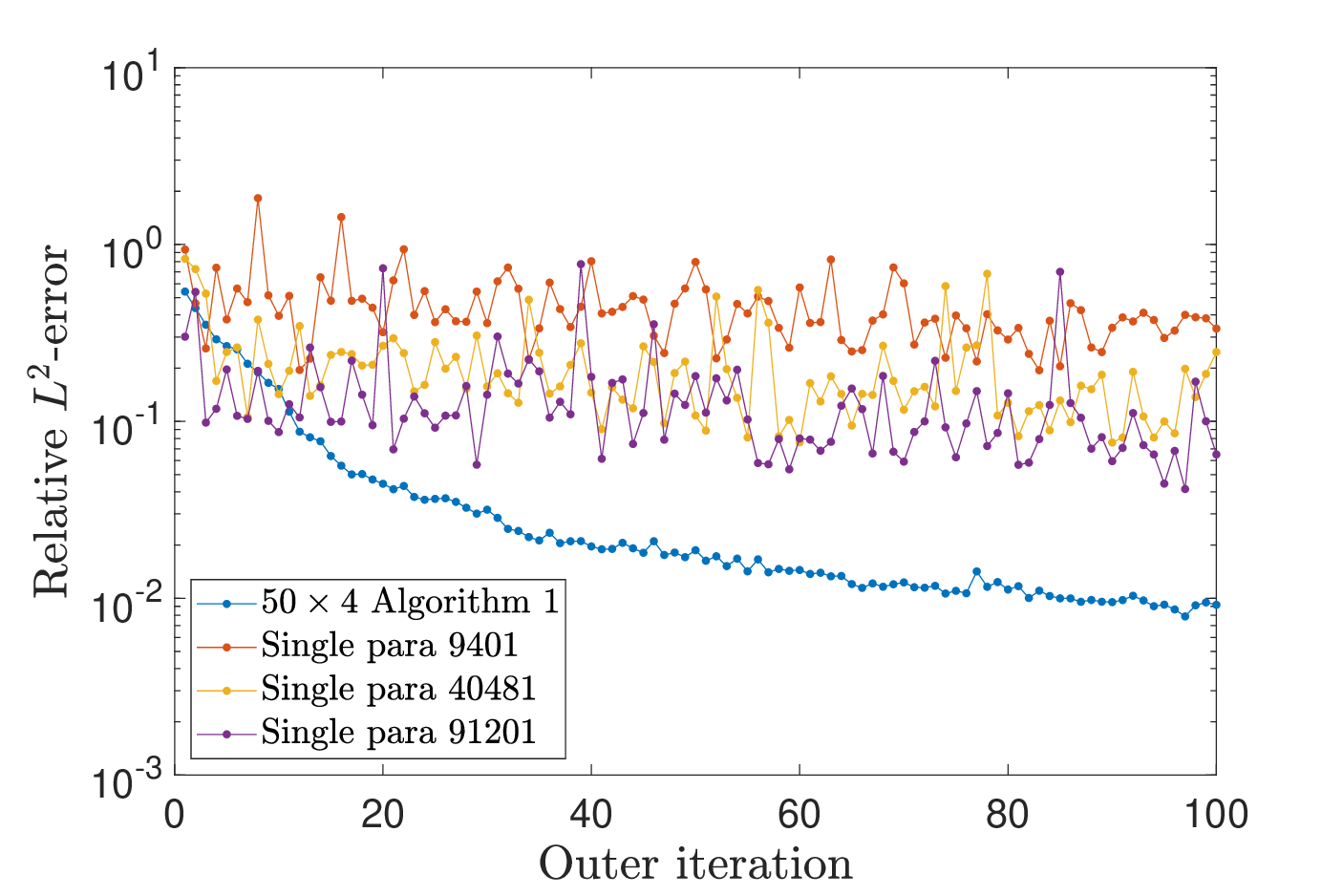}
	\end{center}
	\vskip-.4truecm
	\caption{Large domain oscillatory example in (\ref{model:2d:largedomain}): Error decay history over the outer iterations (and training epochs) in Algorithm 1 (and single-domain case).}\label{fig:2d:large}
	%	\vskip-.2truecm
\end{figure}

{\bf Smooth example.} We now consider a Poisson model problem in a unit square domain
with the following exact solution
%\begin{equation}\label{model:2d:smooth}
%\begin{aligned}
%-\Delta u=f,\quad&\text{in }\,\Omega:=(0,1)\times(0,1),\\
%u=g,\quad &\text{on }\,\partial\Omega,
%\end{aligned}
%\end{equation}
%where $f(x,y)=2\pi^2 \sin(\pi x) \sin(\pi y)$ and $g(x,y)=0$
%to give the exact solution
\begin{equation}\label{model:2d:smooth}
u^*(x,y)=\sin(\pi x)\sin(\pi y),\quad (x,y) \in (0,\, 1)^2.
\end{equation}
We will compare the performance of our proposed iteration methods, Algorithms 1-3, for the smooth example.
In Table~\ref{table:2d:smooth:info}, we list the number of parameters, interior data points, and boundary data points used in our computation.
For all the subdomain partition cases, the same coarse neural network is used in Algorithms 2 and 3.
As increasing the number of subdomains, we set the local network size smaller so as to
make the total sum of parameters comparable for all the subdomain partition cases.
%In Table~\ref{table:2d:smooth:info}, we report the number of parameters, interior data points, and boundary data %points used in our computation. For all $N$ subdomains, the number of parameters and data points used for the coarse %problem is the same as the number used for the local network in $2\times 2$ subdomains.
\begin{table}[ht!]
	\caption{Hyper parameter settings for the computation results in
Table~\ref{table:2d:smooth:error:trained_U0}: The number of layers, neurons, parameters, and the number of interior and boundary training data points for the local and coarse networks (last row).}\label{table:2d:smooth:info}
	{\normalsize \renewcommand{\arraystretch}{1.0}
		\begin{center}
			\vskip-.3truecm
			\begin{tabular}{c|ccccc}
				
				\hline\hline
				& \multicolumn{5}{c}{Local problem} \\
				\hline\hline
				No. of subdomains & Layers & Neurons & Parameters  & Interior data & Boundary data  \\
				\Xhline{3\arrayrulewidth}
				$2\times 2$  & 4 & 27 & 2377 & 1250 & 250  \\
				$4\times 4$  & 4 & 13 & 599 & 313 & 63  \\
				$6\times 6$  & 4 & 9  & 307 & 139 & 28  \\
				$8\times 8$  & 4 & 7  & 197 & 78 & 16  \\
				$10\times 10$& 4 & 5  & 111 & 50 & 10  \\
				\Xhline{0.8\arrayrulewidth}
                Coarse problem & 4 & 27  & 2377 & 1250 & 250   \\
				%Single-domain & 4 & 55  & 9461 & 5000 & 1000   \\
				\Xhline{3\arrayrulewidth}
			\end{tabular}
		\end{center}
	}
	\vskip-.2truecm
\end{table}

For the network used for the initial value setting, we form a fully connected neural network with 2377 parameters and 1500 training data points selected randomly from the entire domain, which include 1,250 interior data points and 250 boundary data points.
We also train the initial network using the Adam optimizer with the learning rate 0.001
up to 10,000 training epochs.

The relative $L^2$-errors between the neural network approximation and the exact solution,
and the total computation times are reported in Table~\ref{table:2d:smooth:error:trained_U0}
for Algorithms 1-3 with a trained $U^{(0)}$.
We report the mean value of errors for the obtained solutions from five different seeds.

The errors in Algorithm 3 are robust to the number of subdomains but those in Algorithms 1 and 2
are getting larger as more subdomains are in the partition.
The total computation times in all the three algorithms are getting smallers as more subdomains are
introduced in the partition to give a smaller size of local neural networks.
In Algorithms 2 and 3, at each outer iteration, we trained each local neural network solution in parallel
and then trained coarse neural network solution.
This caused additional computation time in Algorithms 2 and 3.

By fully parallelizing local and coarse neural network solutions, we can further reduce
the total computation time but due to heterogeneous data and network structures
between the local and coarse problems we simply trained them sequentially
in our computation.
This issue will need further studies on the data and network implementation
on GPU using the JAX library or some other available options.
In Algorithm 3, the total computation times are larger than those in Algorithm 2.
This larger computation time is related to the second derivative evaluation for the right hand side function
corresponding to each local and coarse neural network problems
at every outer iteration, i.e., Step 2 in Algorithm 3.
We note that the second derivative evaluation for the right hand side function is needed
only for the coarse neural network problem in Algorithm 2.

In Fig.~\ref{fig:2d:smmoth:trained_U0}, we present the error decay history of the obtained results
in Table~\ref{table:2d:smooth:error:trained_U0}.
As increasing the number of subdomains, the error decay in Algorithms 1 and 2 is getting slower
while in Algorithm 3 the error decay rates are observed to be robust.
The decay history behaviors in Algorithms 1 and 3 are consistent with our convergence analysis in
Theorem~\ref{cor:conv:one-level:nn:uhat} and Theorem~\ref{cor:conv:two-level:nn:uhat}.
However, the results in Algorithm 2 do not follow our convergence result for the two-level
case in Theorem~\ref{cor:conv:two-level:nn:uhat}. In the following, we will report that
this is related to the error $\epsilon_0$ of the coarse neural network solution in Algorithm 2.
%The coarse problem in Algorithm 2 does not help to accelerate the convergence speed
%as more subdomains are introduced in the partition, contrary to the coarse problem in Algorithm 3.

\begin{table}[ht!]
	\caption{Smooth example in (\ref{model:2d:smooth}): The mean values of the relative $L^2$-errors and the training time in Algorithms 1-3 obtained from 5 different seeds as increasing the number of subdomains with a trained $U^{(0)}$.}\label{table:2d:smooth:error:trained_U0}
	{\normalsize \renewcommand{\arraystretch}{1.0}
		\begin{center}
			\vskip-.3truecm
			\begin{tabular}{ccccccc}

				\hline\hline
				& \multicolumn{2}{c}{Algorithm 1}
				& \multicolumn{2}{c}{Algorithm 2}
				& \multicolumn{2}{c}{Algorithm 3} \\
				\hline\hline
				No. of subdomains & $L^2$-error  & Time (sec) &  $L^2$-error & Time (sec)  &  $L^2$-error & Time (sec) \\
				\Xhline{3\arrayrulewidth}
				Initial $U^{(0)}$ & 0.0118 & 35.9 & 0.0118 & 35.9 & 0.0118 & 35.9 \\
				\Xhline{0.8\arrayrulewidth}
				$2\times 2$    & 0.0021 & 832.9  & 0.0018 & 1020.5 & 0.0012 & 1875.7 \\
				$4\times 4$    & 0.0024 & 528.7  & 0.0017 & 616.6  & 0.0014 & 1088.3 \\
				$6\times 6$    & 0.0047 & 351.8  & 0.0054 & 524.8  & 0.0012 & 899.8  \\
				$8\times 8$    & 0.0060 & 348.1  & 0.0065 & 504.2  & 0.0018 & 862.0  \\
				$10\times 10$  & 0.0104 & 336.4  & 0.0109 & 425.5  & 0.0011 & 778.7  \\
				%\Xhline{0.8\arrayrulewidth}
				%Single-domain  & 0.0011 & 863.6  & 0.0011 & 863.6  & 0.0011 & 863.6  \\
				\Xhline{3\arrayrulewidth}
			\end{tabular}
		\end{center}
	}
	\vskip-.2truecm
\end{table}

\begin{figure}[ht!]
	\begin{center}
		\includegraphics[width=0.3\textwidth]{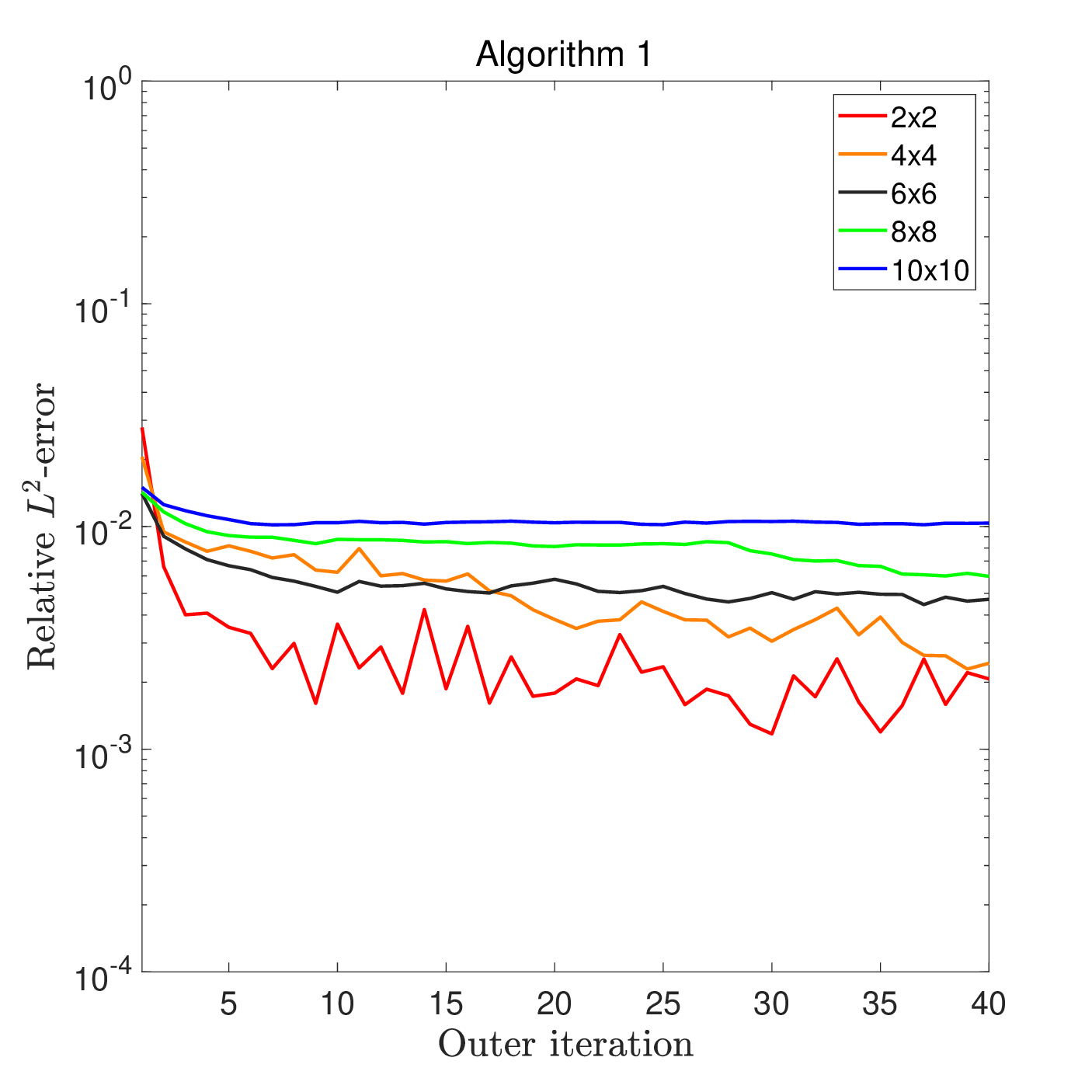} \qquad
		\includegraphics[width=0.3\textwidth]{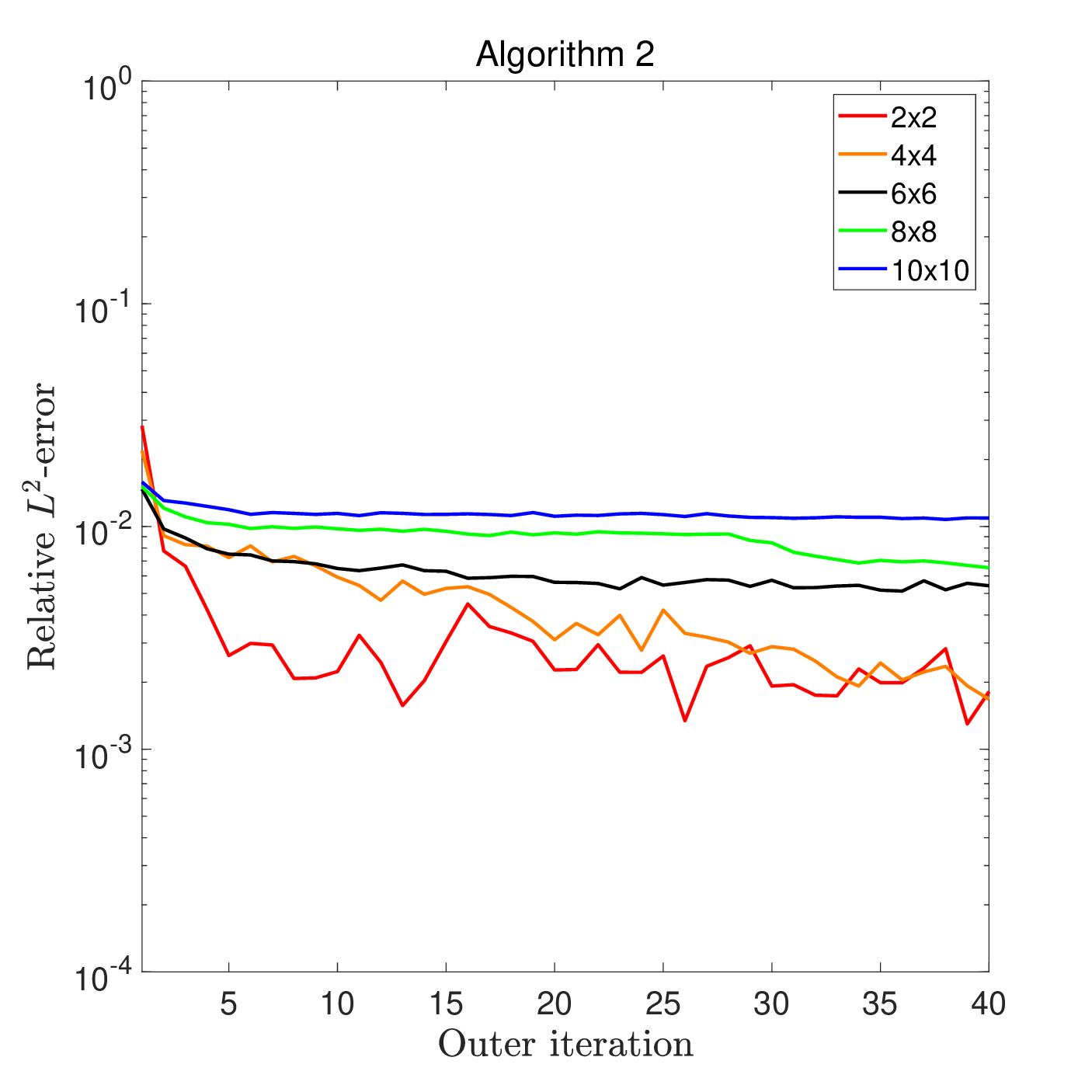} \qquad
		\includegraphics[width=0.3\textwidth]{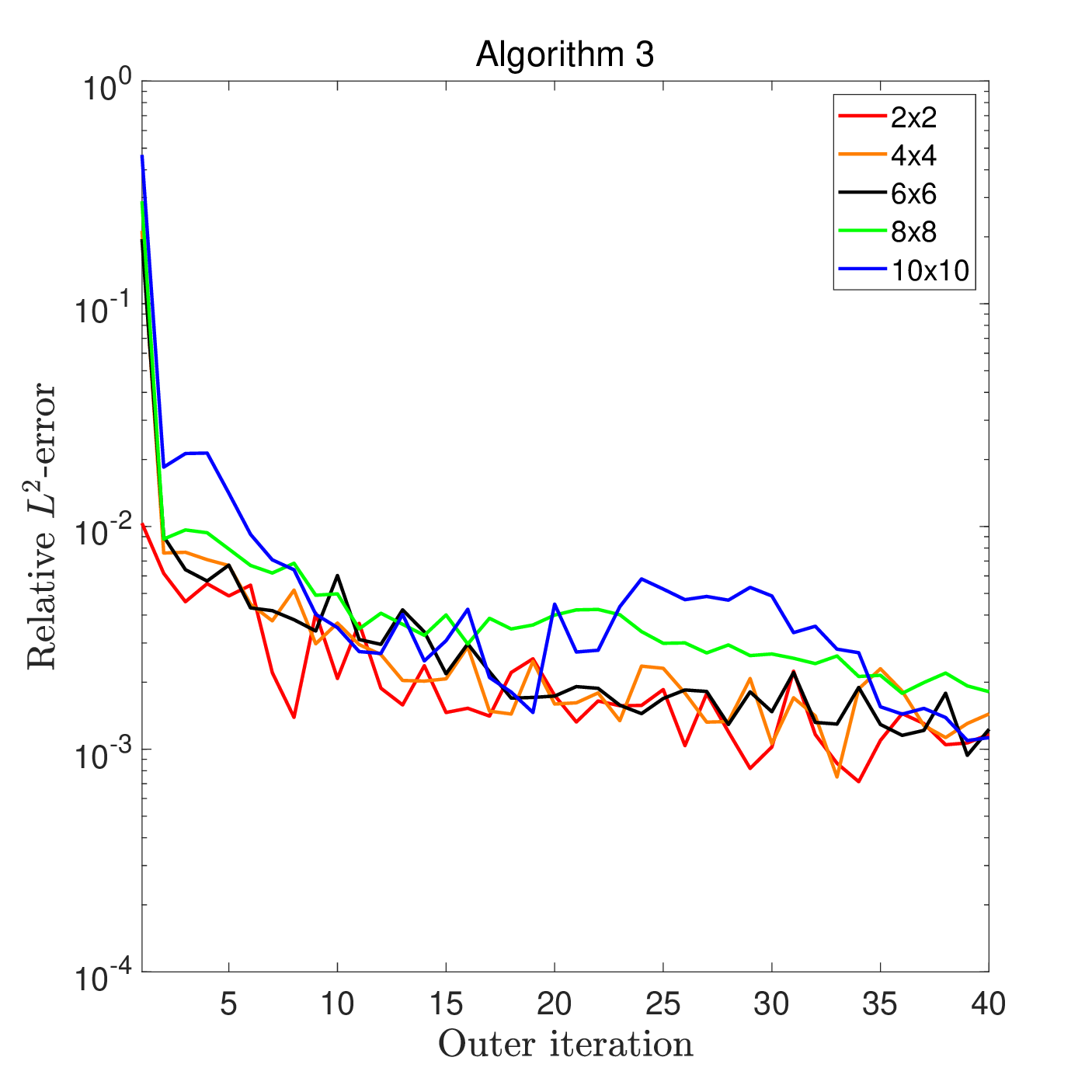}
	\end{center}
	\vskip-.7truecm
	\caption{Smooth example in (\ref{model:2d:smooth}): Error decay history over the outer iterations in the Algorithm 1 (left), Algorithm 2 (middle), and Algorithm 3 (right) as increasing the number of subdomains with a trained $U^{(0)}$.}\label{fig:2d:smmoth:trained_U0}
	%	\vskip-.2truecm
\end{figure}

We now compare the coarse problems in Algorithms 2 and 3 to explain
why the coarse problem in Algorithm 2 does not give scalable convergence results.
In Fig.~\ref{fig:2d:smooth:coarse_forcing},
the forcing terms in the coarse problem of Algorithms 2 and 3 are plotted for the $6\times 6$ subdomain case.
We plot the updated forcing terms after the first outer iteration, i.e., after training the local and coarse problem solutions for the given initial $U^{(0)}$.
In Algorithm 2, the forcing term plot shows oscillatory and high contrast values near the boundary of the overlapping region in the subdomain partition. This is related to the local problems solved in Algorithm 2, where
the local neural network parameters are trained to minimize both the PDE loss and the boundary condition loss.
The boundary condition loss is known to be more difficult to optimize than the PDE loss, see \cite{NTK-theory}.
Such optimization behavior resulted in high variation residual errors
near the overlapping region boundary as seen in Fig.~\ref{fig:2d:smooth:coarse_forcing},
which are difficult to be well-approximated by the coarse neural network with a small enough $\epsilon_0$.
In addition to that, the residual error behaviors are also affected by the network structures
in Algorithms 2 and 3, see \eqref{parallel:two-level:sol} and \eqref{update:Un:two-level} for Algorithm 2
and see\eqref{Uhat:pu} and \eqref{Un:pu} for Algorithm 3.
In Algorithm 3, by utilizing the partition of unity functions, the local problems for the interior subdomains,
that do not touch the model problem boundary, only have the PDE loss term.
Such local problems and the network structure in \eqref{Uhat:pu} lead to smooth residual errors 
that can be well-resolved by the coarse neural network.
The coarse neural network solution in Algorithm 3 can thus be obtained with a small enough error $\epsilon_0$
to the corresponding coarse Hilbert space solution to give scalable convergence results
as proven in Theorem~\ref{cor:conv:two-level:nn:uhat}.

\begin{figure}[ht!]
	\begin{center}
		\includegraphics[width=0.45\textwidth]{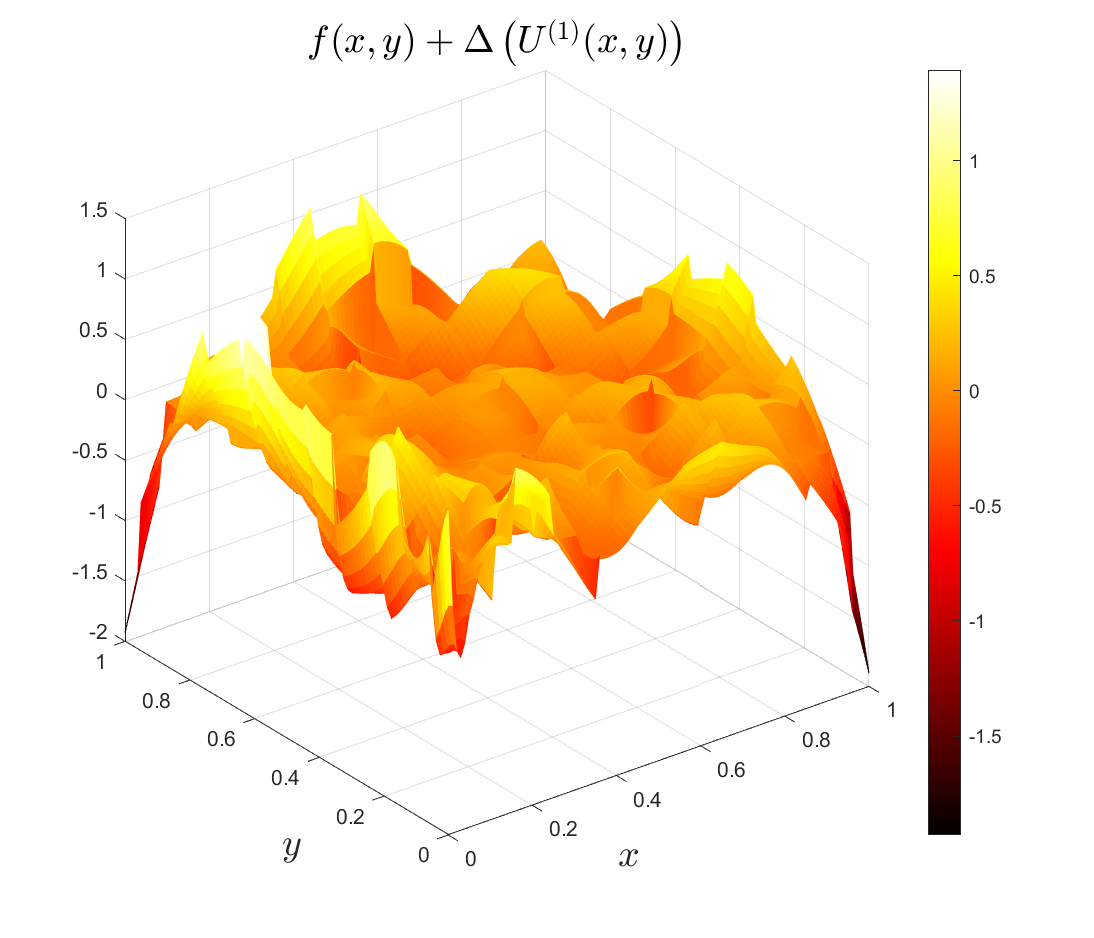} \qquad
		\includegraphics[width=0.45\textwidth]{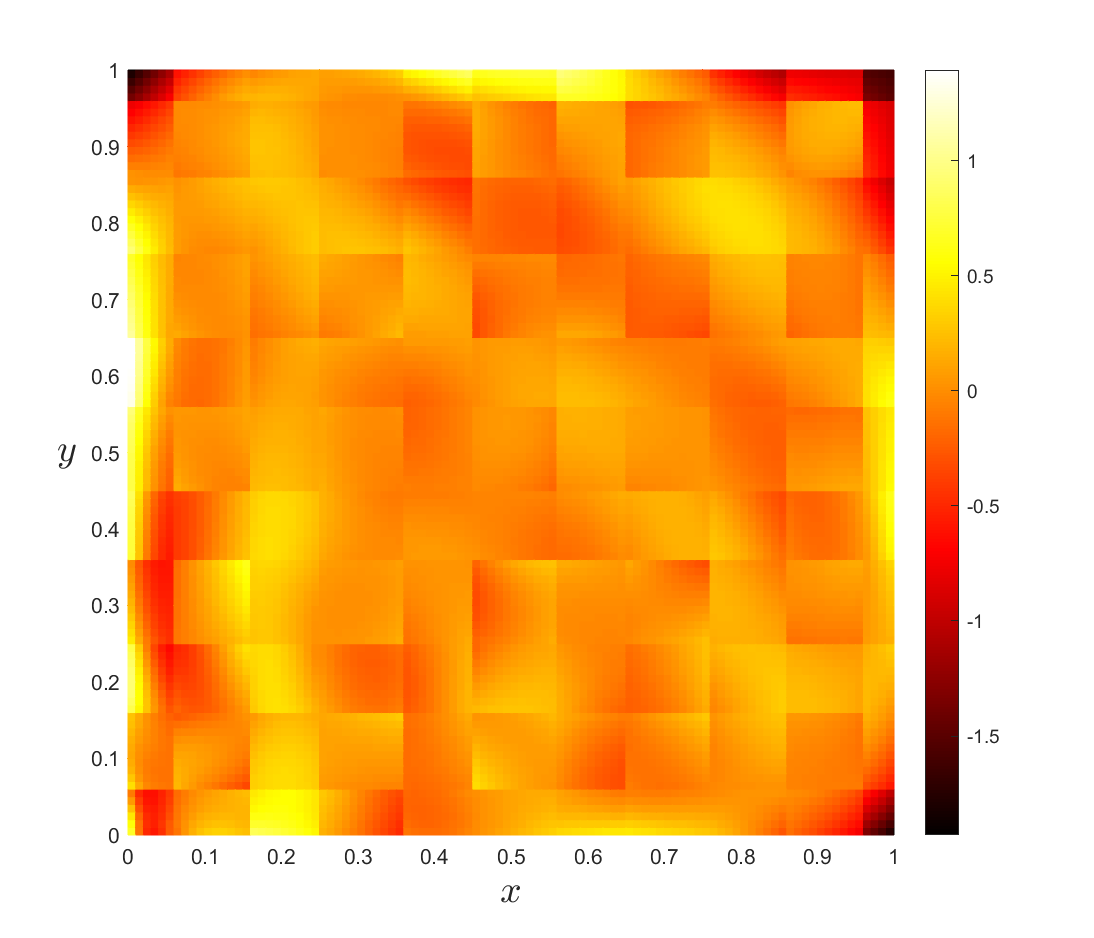} \\
		\includegraphics[width=0.45\textwidth]{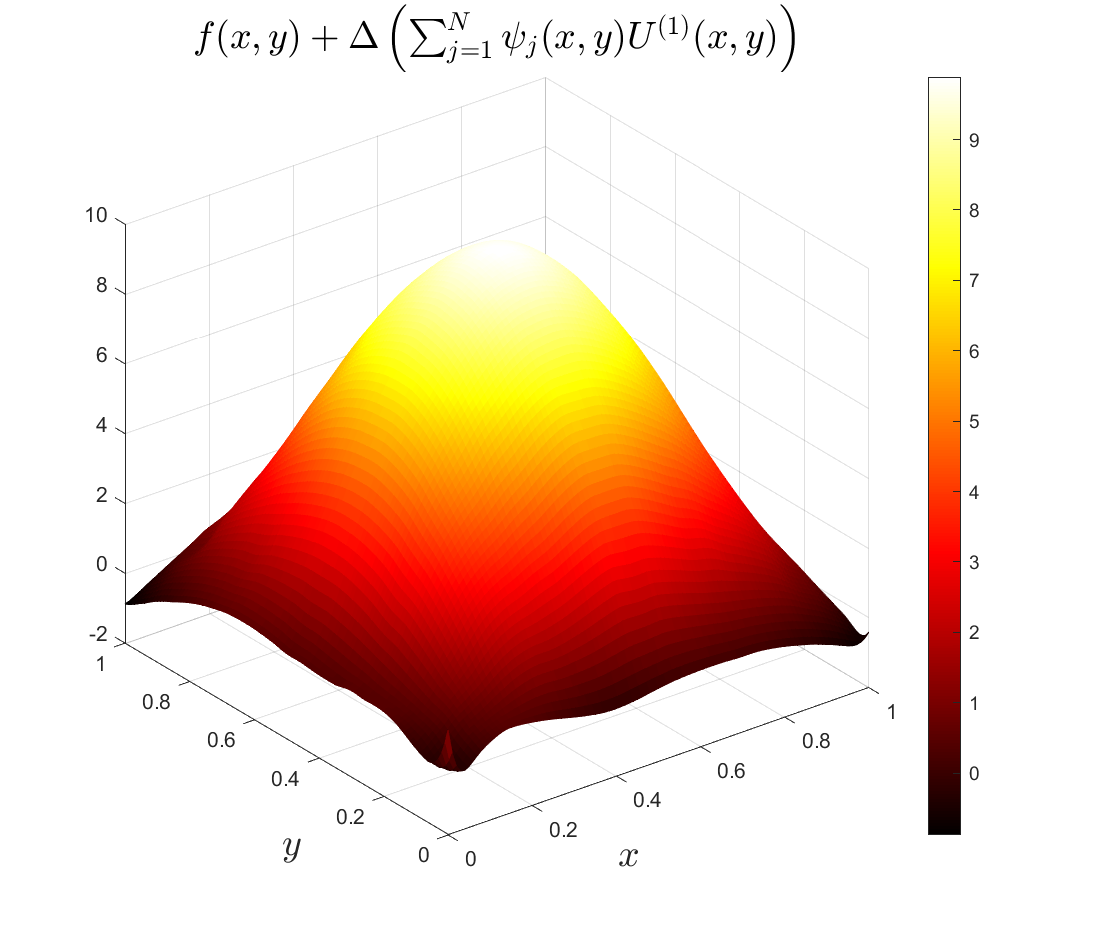} \qquad
		\includegraphics[width=0.45\textwidth]{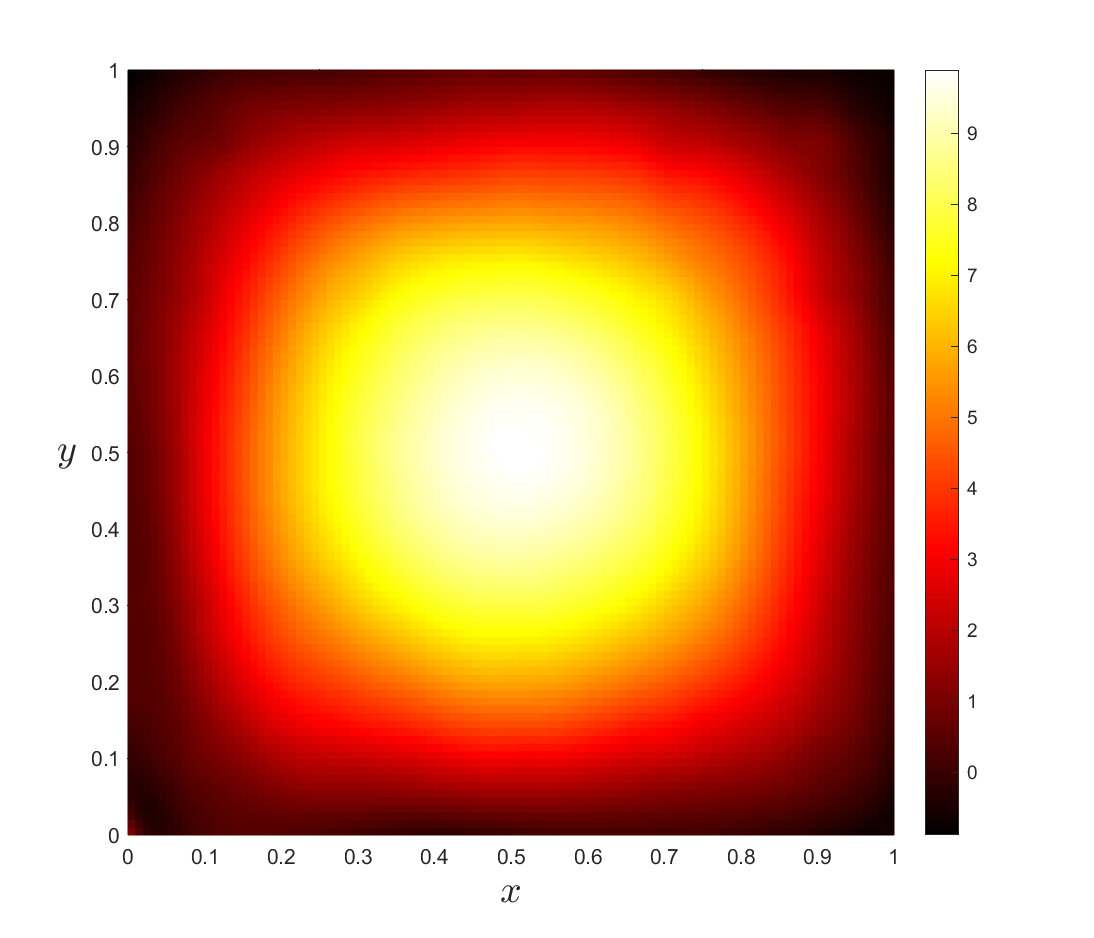}
	\end{center}
	\vskip-.7truecm
	\caption{Forcing terms in the coarse problem: Algorithm 2 (Mesh plot (top left) and Contour plot (top right)),
Algorithm 3 (Mesh plot (bottom left) and Contour plot (bottom right)).}\label{fig:2d:smooth:coarse_forcing}
	%	\vskip-.2truecm
\end{figure}

%{\bf{High contrast and oscillatory examples.}}
%We now consider the Poisson problem with a high contrast and oscillatory solution,
%\begin{equation}\label{model:2d:high-cont-osc}
%u^*(x,y)=Ax(1-x)y(1-y)\sin\left(\frac{(x-0.5)(y-0.5)}{\epsilon}\right),
%\end{equation}
%where $A$ and $\epsilon$ are parameters to be chosen.

{\bf{Multiscale example.}} We now consider the Poisson problem with a multiscale solution,
\begin{align}\label{model:2d:multiscale}
	u^*(x,y)=\sin(\pi x)\sin(\pi y)+0.05\sin(8\pi x)\sin(8\pi y), \quad (x,y)\in(0,1)^2
\end{align}
and test the performance of our iterative algorithms for this more challenging model problem.
In Fig.~\ref{fig:2d:multiscale:exact}, plots of the exact solution and the corresponding right hand side function
in the Poisson problem are presented and in Table~\ref{table:2d:multiscale:info},
the hyper parameter settings used in the following numerical results are listed.

\begin{figure}[ht!]
	\centering
	\includegraphics[width=0.4\textwidth]{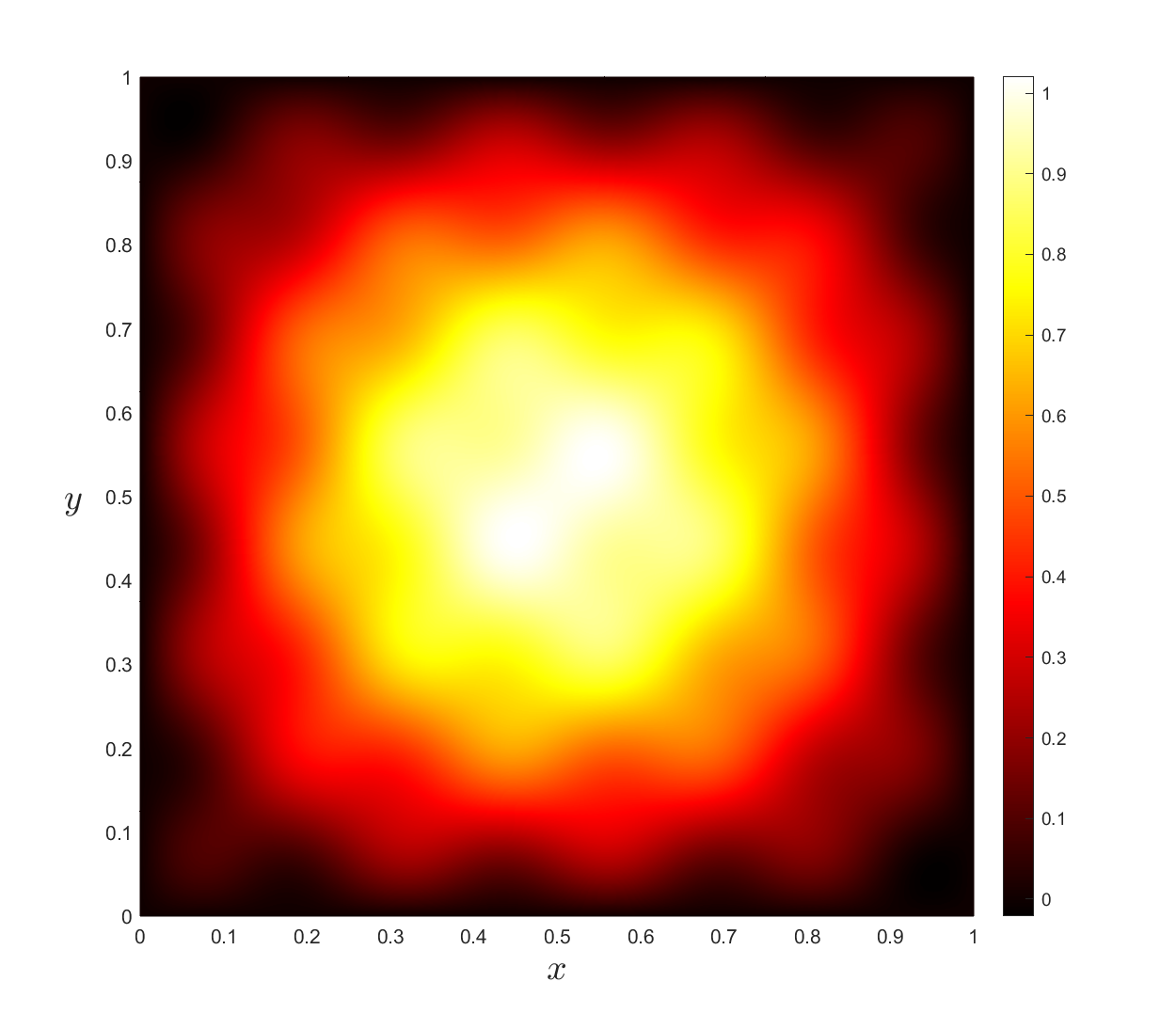}
	\includegraphics[width=0.4\textwidth]{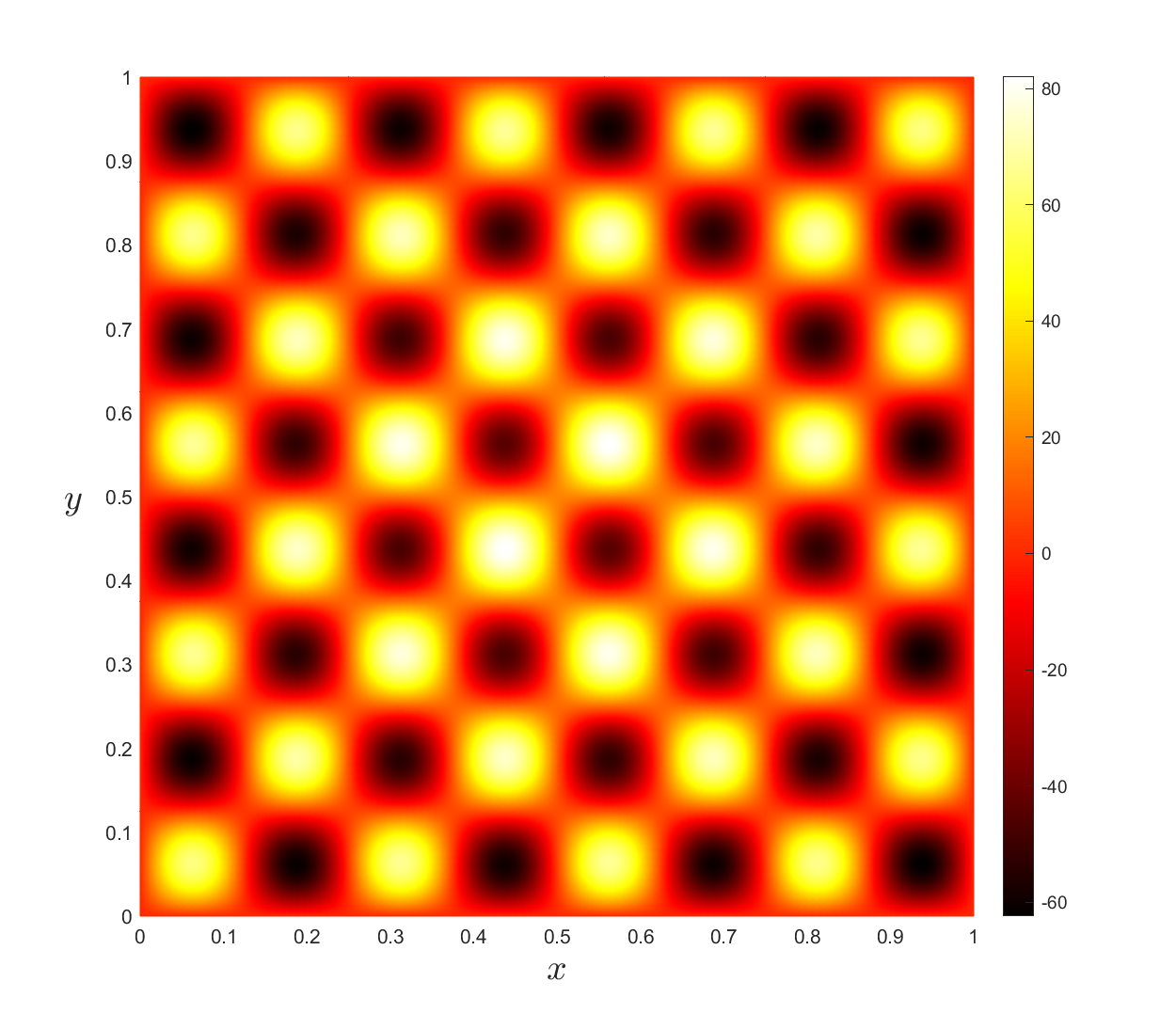}
	\vskip-.4truecm
	\caption{Multiscale example in (\ref{model:2d:multiscale}): Exact solution $u^*(x,y)$ (left) and forcing term $f(x,y)$ (right)}
	\label{fig:2d:multiscale:exact}
\end{figure}

\begin{table}[ht!]
	\caption{Hyper parameter settings for the computation results in
Table~\ref{table:2d:multiscale:error:trained_U0}: The number of layers, neurons, parameters, and the number of interior and boundary training data points for the local and coarse networks (last row).}\label{table:2d:multiscale:info}
	{\normalsize \renewcommand{\arraystretch}{1.0}
		\begin{center}
			\vskip-.3truecm
			\begin{tabular}{c|ccccc}
				
				\hline\hline
				& \multicolumn{5}{c}{Local problem} \\
				\hline\hline
				No. of subdomains & Layers & Neurons & Parameters  & Interior data & Boundary data  \\
				\Xhline{3\arrayrulewidth}
				$2\times 2$  & 4 & 28 & 2549  & 1250 & 250  \\
				$4\times 4$  & 3 & 16 & 609   & 313  & 63   \\
				$6\times 6$  & 2 & 14 & 267   & 139  & 28   \\
				$8\times 8$  & 2 & 10 & 151   & 78   & 16   \\
				\Xhline{0.8\arrayrulewidth}
				%Single-domain& 4 & 57 & 10147 & 5000 & 1000 \\
                Coarse problem & 4 & 28 & 2549 & 1250 & 250 \\
				\Xhline{3\arrayrulewidth}
			\end{tabular}
		\end{center}
	}
	\vskip-.2truecm
\end{table}

Table~\ref{table:2d:multiscale:error:trained_U0} shows the relative $L^2$-errors for the neural network approximation obtained from Algorithms 1-3, with a trained $U^{(0)}$ as the initial.
For the initial $U^{(0)}$, we use the same network size and training settings as in the previous smooth example.
In this multiscale example, we observe similar error behaviors as in the smooth example.
In Algorithms 1 and 2, the iteration convergence gets slower as partitioning the problem domain with more subdomains
while in Algorithm 3, the error results are maintained at a similar level for all the subdomain partition cases.
In Algorithms 1 and 2, for the $6\times6$ and $8\times8$ subdomain cases the resulting errors
are even larger than that in the initial $U^{(0)}$.
It shows that as partitioning the problem domain into more subdomains,
the solution convergence in Algorithms 1 and 2 may get quite badly, even started with a good initial.
The use of the partition of unity functions in Algorithm 3 can resolve
this convergence problem in Algorithm 1 to give robust error results for all the subdomain partition cases.
We note that the only difference between Algorithms 2 and 3 is the use of the partition of unity functions
when forming the partitioned neural networks.
The coarse problem in Algorithm 3 works effectively to solve the convergence problem observed in Algorithm 1
as increasing the number of subdomains, since the local neural network solutions in Algorithm 3
produce a smooth right hand side function in the coarse problem that can
be approximated well enough by the coarse neural network.
In Fig.~\ref{fig:2d:multiscale}, error decay history for the obtained results in Table~\ref{table:2d:multiscale:error:trained_U0} is presented.
In Algorithms 1 and 2, the error decay rates depend on the number of subdomains in the partition, i.e.,
as the more subdomains in the partition, the slower the error decay rates.
In Algorithm 3, the coarse problem works effectively to give faster error decay rates
than those in Algorithms 1 and 2.

\begin{table}[ht!]
\caption{Multiscale example in (\ref{model:2d:multiscale}): The mean values of the relative $L^2$-errors in the Algorithms 1-3 obtained from 5 different seeds as increasing the number of subdomains with a trained $U^{(0)}$.}\label{table:2d:multiscale:error:trained_U0}
	{\normalsize \renewcommand{\arraystretch}{1.0}
		\begin{center}
			\vskip-.3truecm
			\begin{tabular}{cccc}

				\hline\hline
				No. of subdomains & Algorithm 1 & Algorithm 2 & Algorithm 3  \\
				\Xhline{3\arrayrulewidth}
				Initial $U^{(0)}$ & 0.0131 & 0.0131 & 0.0131  \\
				\Xhline{0.8\arrayrulewidth}
				$2\times 2$    & 0.0021 & 0.0021 & 0.0023  \\
				$4\times 4$    & 0.0040 & 0.0034 & 0.0026  \\
				$6\times 6$    & 0.0192 & 0.0167 & 0.0017  \\
				$8\times 8$    & 0.0407 & 0.0451 & 0.0033  \\
				%\Xhline{0.8\arrayrulewidth}
				%Single-domain & 0.0020 & 0.0020 & 0.0020  \\
				\Xhline{3\arrayrulewidth}
			\end{tabular}
		\end{center}
	}
	\vskip-.2truecm
\end{table}

\begin{figure}[ht!]
	\begin{center}
		\includegraphics[width=0.3\textwidth]{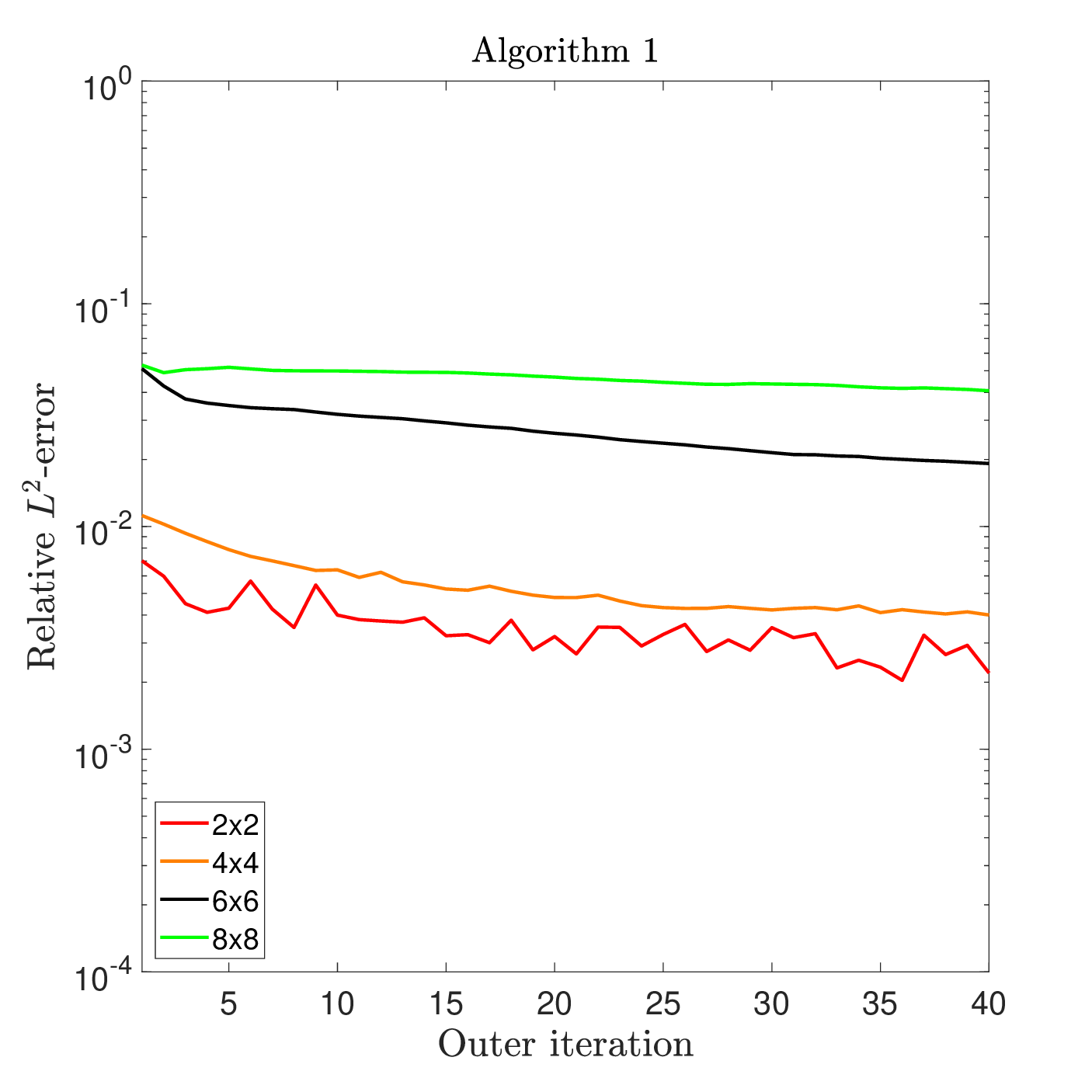} \qquad
		\includegraphics[width=0.3\textwidth]{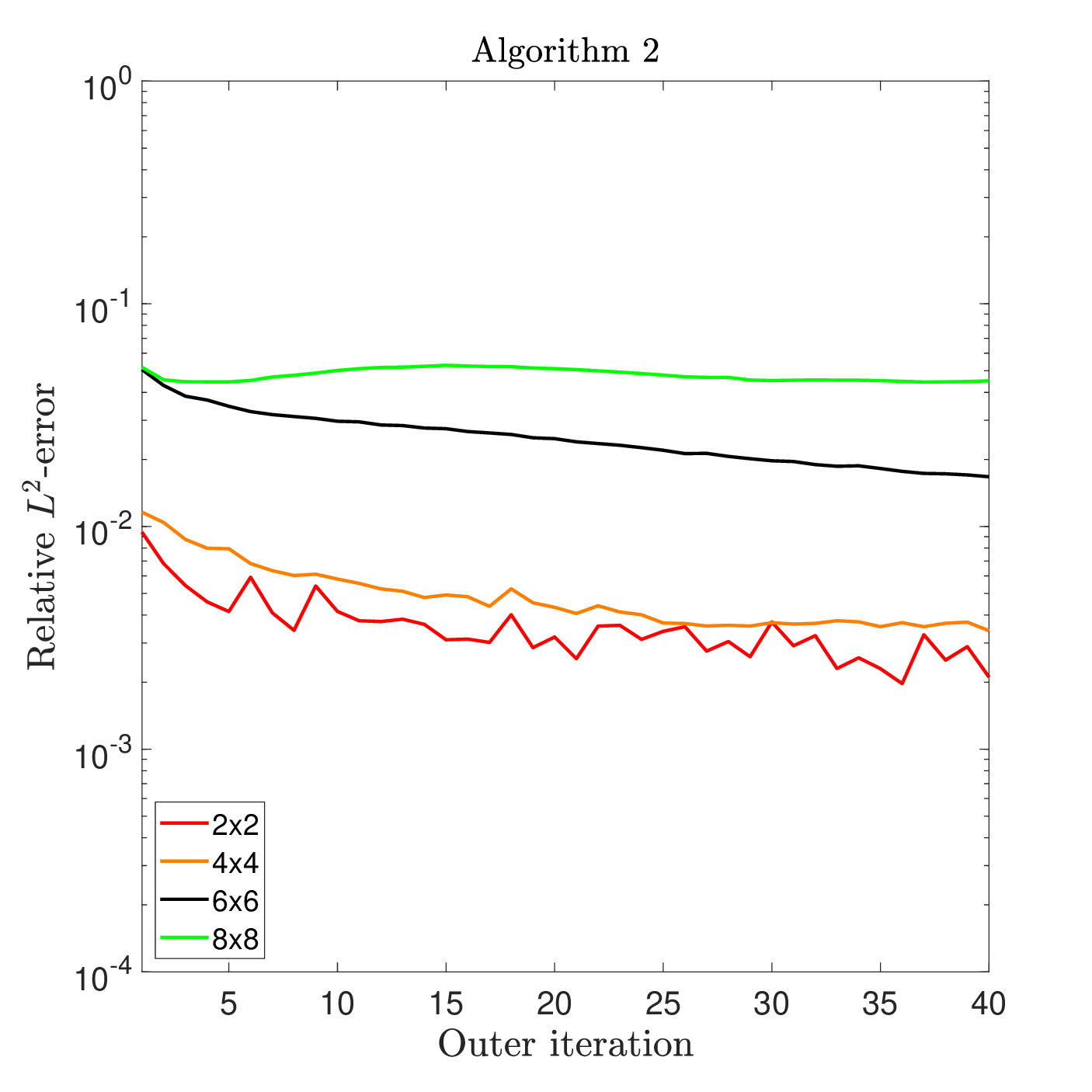} \qquad
		\includegraphics[width=0.3\textwidth]{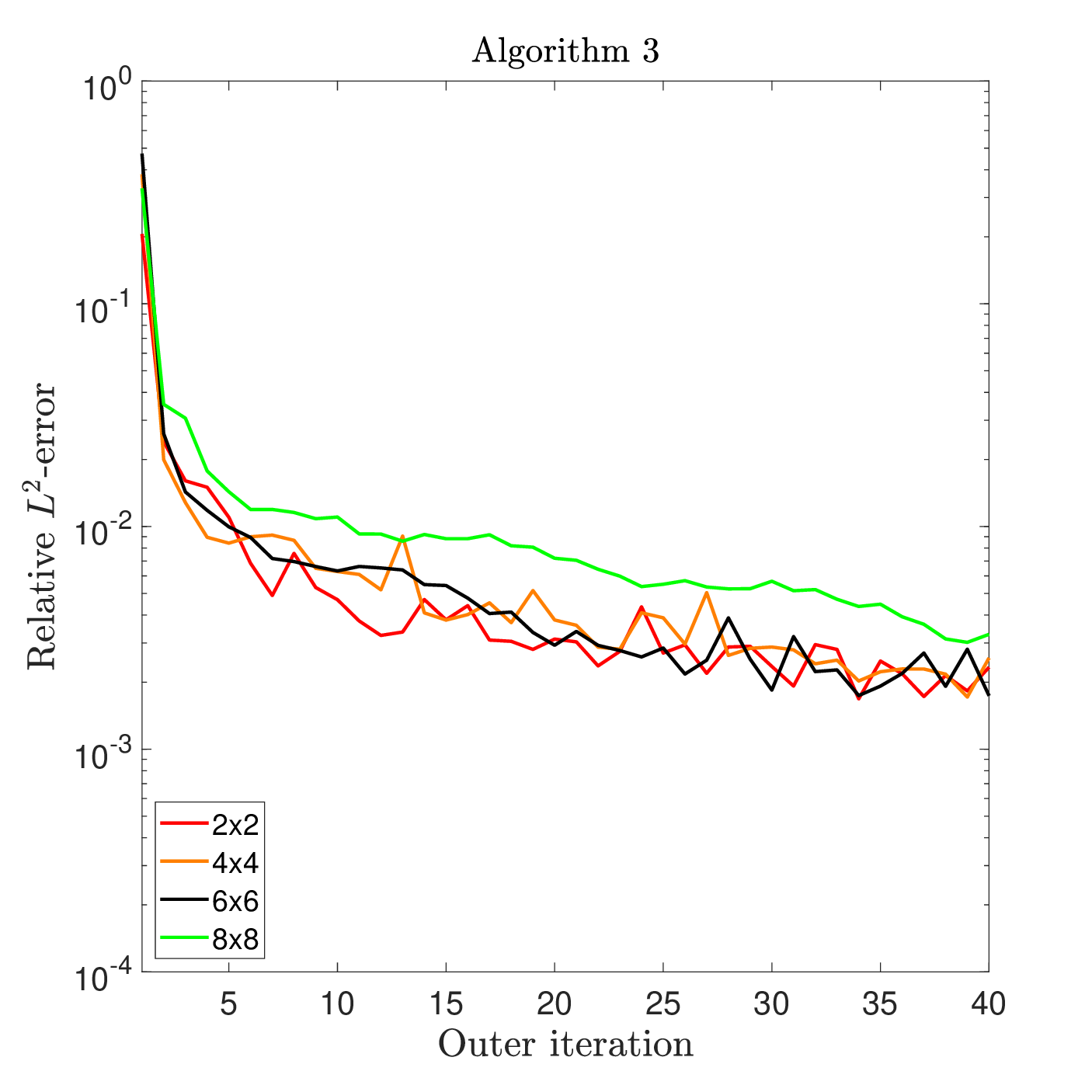}
	\end{center}
	\vskip-.7truecm
	\caption{Multiscale example in (\ref{model:2d:multiscale}):
Error decay history over the outer iterations in Algorithm 1 (left), Algorithm 2 (middle), and Algorithm 3 (right) as increasing the number of subdomains with a trained $U^{(0)}$.}\label{fig:2d:multiscale}
	%	\vskip-.2truecm
\end{figure}

%\subsection{Validation of Assumption~\ref{assume3}}

\section{Conclusions}\label{sec:conclude}
Iterative algorithms are proposed and analyzed for partitioned neural network approximation
defined on overlapping subdomain partitions of the problem domain,
with the aim of reducing the heavy communication cost in the partitioned
neural network approach.
The algorithm development and analysis are based on the classical additive Schwarz method
and the error assumptions on the local and coarse neural network solutions at each iteration.
To enhance the scalability of the proposed algorithms as increasing the number of subdomains, 
a two-level algorithm with partitioned neural networks, that are formed by partition of unity functions,
is proposed.
Such a partitioned network structure produces favorable residual errors in the coarse problem 
at each iteration, that can be well resolved by the coarse neural network solution,
and it thus makes the coarse neural network solution work effectively
to speed up the iteration convergence as increasing the number of subdomains.

\section*{Acknowledgments}
The second author was supported by the National Research Foundation of Korea(NRF) grants
funded by NRF-2022R1A2C100388511.

%%%%%%%%%%%%%%%%%%%%%%%%%%%%%%

% \section*{References}

% Please ensure that every reference cited in the text is also present in
% the reference list (and vice versa).

% \section*{\itshape Reference style}

% Text: All citations in the text should refer to:
% \begin{enumerate}
% \item Single author: the author's name (without initials, unless there
% is ambiguity) and the year of publication;
% \item Two authors: both authors' names and the year of publication;
% \item Three or more authors: first author's name followed by `et al.'
% and the year of publication.
% \end{enumerate}
% Citations may be made directly (or parenthetically). Groups of
% references should be listed first alphabetically, then chronologically.

%%Vancouver style references.
\bibliographystyle{plain}
\bibliography{ddforpinn}

%\section*{Supplementary Material}
%
%Supplementary material that may be helpful in the review process should
%be prepared and provided as a separate electronic file. That file can
%then be transformed into PDF format and submitted along with the
%manuscript and graphic files to the appropriate editorial office.

\end{document}